\documentclass[11pt]{article}
\usepackage[english]{babel}
\usepackage{amssymb,amsmath,amsthm,graphicx,ytableau}
\usepackage[a4paper,margin=2.54cm]{geometry}

\DeclareMathOperator{\Wr}{Wr}
\DeclareMathOperator{\htt}{ht}

\usepackage{tikz}
\usetikzlibrary{decorations.pathreplacing}

\newtheorem{theorem}{Theorem}[section]
\newtheorem{lemma}[theorem]{Lemma}
\newtheorem{proposition}[theorem]{Proposition}

\theoremstyle{definition}
\newtheorem{example}[theorem]{Example}
\newtheorem{remark}[theorem]{Remark}

\usepackage{hyperref}
\hypersetup{colorlinks, citecolor=black, linkcolor=black}
\AtBeginDocument{} 

\numberwithin{equation}{section}

\author{Niels Bonneux}
\date{\today}
\title{Asymptotic behavior of Wronskian polynomials \\ that are factorized via $p$-cores and $p$-quotients}

\makeatletter
\let\save@mathaccent\mathaccent
\newcommand*\if@single[3]{
	\setbox0\hbox{${\mathaccent"0362{#1}}^H$}%
	\setbox2\hbox{${\mathaccent"0362{\kern0pt#1}}^H$}%
	\ifdim\ht0=\ht2 #3\else #2\fi
}
\newcommand*\rel@kern[1]{\kern#1\dimexpr\macc@kerna}
\newcommand*\widebar[1]{\@ifnextchar^{{\wide@bar{#1}{0}}}{\wide@bar{#1}{1}}}
\newcommand*\wide@bar[2]{\if@single{#1}{\wide@bar@{#1}{#2}{1}}{\wide@bar@{#1}{#2}{2}}}
\newcommand*\wide@bar@[3]{%
	\begingroup
	\def\mathaccent##1##2{%
		\let\mathaccent\save@mathaccent
		\if#32 \let\macc@nucleus\first@char \fi
		\setbox\z@\hbox{$\macc@style{\macc@nucleus}_{}$}%
		\setbox\tw@\hbox{$\macc@style{\macc@nucleus}{}_{}$}%
		\dimen@\wd\tw@
		\advance\dimen@-\wd\z@
		\divide\dimen@ 3
		\@tempdima\wd\tw@
		\advance\@tempdima-\scriptspace
		\divide\@tempdima 10
		\advance\dimen@-\@tempdima
		\ifdim\dimen@>\z@ \dimen@0pt\fi
		\rel@kern{0.6}\kern-\dimen@
		\if#31
		\overline{\rel@kern{-0.6}\kern\dimen@\macc@nucleus\rel@kern{0.4}\kern\dimen@}%
		\advance\dimen@0.4\dimexpr\macc@kerna
		\let\final@kern#2%
		\ifdim\dimen@<\z@ \let\final@kern1\fi
		\if\final@kern1 \kern-\dimen@\fi
		\else
		\overline{\rel@kern{-0.6}\kern\dimen@#1}%
		\fi
	}%
	\macc@depth\@ne
	\let\math@bgroup\@empty \let\math@egroup\macc@set@skewchar
	\mathsurround\z@ \frozen@everymath{\mathgroup\macc@group\relax}%
	\macc@set@skewchar\relax
	\let\mathaccentV\macc@nested@a
	\if#31
	\macc@nested@a\relax111{#1}%
	\else
	\def\gobble@till@marker##1\endmarker{}%
	\futurelet\first@char\gobble@till@marker#1\endmarker
	\ifcat\noexpand\first@char A\else
	\def\first@char{}%
	\fi
	\macc@nested@a\relax111{\first@char}%
	\fi
	\endgroup
}

\begin{document}
\maketitle
	
\begin{abstract}
In this paper we consider Wronskian polynomials labeled by partitions that can be factorized via the combinatorial concepts of $p$-cores and $p$-quotients. We obtain the asymptotic behavior for these polynomials when the $p$-quotient is fixed while the size of the $p$-core grows to infinity. For this purpose, we associate the $p$-core with its characteristic vector and let all entries of this vector simultaneously tend to infinity. This result generalizes the Wronskian Hermite setting which is recovered when $p=2$. 
\end{abstract}

\section{Introduction}\label{sec:Introduction}
In this paper we study the asymptotic behavior of the Wronskian polynomials
\begin{equation}\label{eq:IntroQLambda}
	q_\lambda 
		= \frac{\Wr[q_{n_1},q_{n_2},\dots,q_{n_{r}}]}{\Delta(n_{\lambda})}
		= \frac{\det \left(\frac{d^{i-1}}{dx^{i-1}} \, q_{n_j}\right)_{1\leq i ,j \leq r}}{\prod\limits_{i<j}(n_j-n_i)}
\end{equation}
which are labeled by partitions $\lambda=(\lambda_1,\lambda_2,\dots,\lambda_r)$ and where the associated degree vector $n_\lambda=(n_1,n_2,\dots,n_r)$, defined by~$n_i=\lambda_i-i+r$ for~$i=1,2,\dots,r$, describes the degrees of the polynomial entries in the Wronskian. These entries are defined via the exponential generating function
\begin{equation}\label{eq:IntroExponentialGeneratingFunction}
	\sum_{n=0}^\infty q_n(x) \frac{t^n}{n!} 
		= \exp\left(tx - \frac{t^p}{p}\right)
\end{equation}
where~$p$ is a fixed positive integer. The Wronskian in~\eqref{eq:IntroQLambda} is divided by the Vandermonde determinant of the degree vector, i.e., $\Delta(n_{\lambda})=\prod_{i<j}(n_j-n_i)$, and acts as a normalization constant to obtain monic polynomials. The Wronskian polynomials~\eqref{eq:IntroQLambda} appear in several applications for specific values of~$p$, most notably for $p=2$ and $p=3$. For~$p=1$, the set-up trivializes.

When~$p=2$, we get Wronskians of Hermite polynomials which are of interest as they describe the rational solutions of the fourth Painlev\'e equation and its higher order generalizations~\cite{Clarkson-zeros,Clarkson_GomezUllate_Grandati_Milson,Kajiwara_Ohta-PIV,Noumi_Yamada}. The asymptotic behavior of the polynomial zeros associated with Painlev\'e IV is studied in several works~\cite{Buckingham,Masoero_Roffelsen,Masoero_Roffelsen-2019}. Another domain where these Wronskian Hermite polynomials appear, is in the field of exceptional orthogonal polynomials. A description of exceptional Hermite polynomials, labeled by partitions, is given in~\cite{GomezUllate_Grandati_Milson-Hermite}, their zeros are studied in~\cite{Kuijlaars_Milson}, and recurrence relations are obtained in~\cite{Bonneux_Stevens,Duran-Recurrence,GomezUllate_Kasman_Kuijlaars_Milson}. 

When~$p=3$, the polynomials~$(q_n)_n$ are multiple orthogonal polynomials~\cite{VanAssche}, and the net of polynomials~$(q_\lambda)_{\lambda}$ comprises the Yablonskii-Vorobiev polynomials~\cite{Kajiwara_Ohta}. They are associated with the rational solutions of the second Painlev\'e equation and their zeros are studied in~\cite{Bertola_Bothner,Buckingham_Miller-1,Buckingham_Miller-2}. 

In this paper, we do not specify the value of~$p$ and merely require it to be a positive integer. For such general~$p$, the polynomial sequence~$(q_n)_{n}$ is an Appell sequence and the combinatorial framework which associates a Wronskian polynomial~$q_\lambda$ defined as in~\eqref{eq:IntroQLambda} with every partition, was dubbed \textit{Wronskian Appell polynomials} in~\cite{Bonneux_Hamaker_Stembridge_Stevens}. An appealing feature of this combinatorial framework is the existence of a recurrence relation that generates the whole net of polynomials. This relation becomes particularly remarkable for the Appell sequences~$(q_n)_n$ that we consider in this paper, i.e., those with a generating function as in~\eqref{eq:IntroExponentialGeneratingFunction}, see~\cite[Section~7.2]{Bonneux_Hamaker_Stembridge_Stevens}. The particular form of this recurrence relation was exploited in~\cite{Bonneux_Dunning_Stevens} to show that every polynomial~$q_\lambda$ factorizes according to the combinatorial concepts of~$p$-cores and $p$-quotients~\cite{Garvan_Kim_Stanton,James_Kerber,MacDonald} for the same~$p$ as in~\eqref{eq:IntroExponentialGeneratingFunction}. The $p$-core is the unique partition obtained by removing as many border strips of size~$p$ as possible from the original partition, its size describes the multiplicity of the origin as a zero of~$q_\lambda$. The $p$-quotient, which is an ordered $p$-tuple of partitions obtained via $p$-modular decomposition of the original partition, contains all other information about the coefficients of the polynomial. More concretely, Theorem~6 in~\cite{Bonneux_Dunning_Stevens} states that
\begin{equation}\label{eq:IntroFactorization}
	q_\lambda(x) =
		x^{|\bar{\lambda}|} R_{\lambda}(x^p)
\end{equation}
where $\bar{\lambda}$ denotes the $p$-core of~$\lambda$, and the polynomial~$R_{\lambda}$ can be described via the $p$-quotient of~$\lambda$. In the same paper, the asymptotic behavior of these polynomials is studied for~$p=2$: fix a $2$-quotient $(\mu^{(0)},\mu^{(1)})$, write the $2$-core as the partition $(k,k-1,\dots,2,1)$, and let~$k$ tend to infinity, then the asymptotic behavior for the associated polynomials is given by
\begin{equation}\label{eq:AsymptoticBehavior2}
	\lim_{k \to + \infty} \frac{R_{\lambda(k)}(2kx)}{(2k)^{|\mu^{(0)}|+|\mu^{(1)}|}}
	= (x+1)^{|\mu^{(0)}|} \, (x-1)^{|\mu^{(1)}|}
\end{equation}
where the partition~$\lambda(k)$ is identified by the fixed $2$-quotient $(\mu^{(0)},\mu^{(1)})$ and the $k$-depending $2$-core $(k,k-1,\dots,2,1)$, see Theorem~5 in~\cite{Bonneux_Dunning_Stevens}. We observe that, as~$k$ grows, the size $|\mu^{(0)}|$ (respectively~$|\mu^{(1)}|$) indicates the number of zeros that are attracted by~$-1$ (respectively $1$).

We now generalize this asymptotic result to arbitrary~$p$. To do so, we describe the $p$-core via a $p$-tuple of integers $(c_0,c_1,...c_{p-1})$ called the characteristic vector~\cite{Brunat_Nath,Garvan_Kim_Stanton}. Next we let this vector depend on $k\in\mathbb{N}$, i.e., $c(k)=(c_0(k), c_1(k),\dots,c_{p-1}(k))$, such that all entries grow linearly at most: for each~$i$ we require that $c_i(k)= a_i k + o(k)$ as $k\to+\infty$, with $a_i\in\mathbb{R}$. Then we identify the partition~$\lambda(k)$ as the unique partition with $p$-quotient $\mu=(\mu^{(0)},\mu^{(1)},\dots,\mu^{(p-1)})$ and $p$-core described by the characteristic vector~$c(k)$. Finally, we let $k$ tend to infinity and obtain the asymptotic behavior for the associated polynomials $R_{\lambda(k)}$ defined in~\eqref{eq:IntroFactorization}. 

\paragraph{Main result.}
Fix an integer $p\geq1$ and a $p$-quotient~$\mu$. For any $k\in\mathbb{N}$, let $\lambda(k)$ be the partition identified via its characteristic vector $c(k)$ and its $p$-quotient~$\mu$. If $c_i(k)= a_i k + o(k)$ as $k\to+\infty$ for all $i$, then the polynomials~$R_{\lambda(k)}$ satisfy
\begin{equation}\label{Intro:Result}
	\lim_{k\to + \infty} \frac{R_{\lambda(k)}\left((pk)^{p-1}x\right)}{(pk)^{(p-1)|\mu|}}
		= \prod_{i=0}^{p-1} (x-\alpha_i)^{|\mu^{(i)}|}
\end{equation}
where~$\alpha_i=\prod_{j\neq i} (a_i-a_j)$ for all~$i$. \\

In the described set-up, we let all entries of the characteristic vector simultaneously tend to infinity and obtain that, as~$k$ grows, $|\mu^{(i)}|$~zeros are attracted by the real value~$\alpha_i$. We reformulate this main result in Theorem~\ref{thm:AsymptoticResult}.

The paper is constructed as follows. In Section~\ref{sec:Preliminaries} we first introduce the necessary combinatorial concepts and then use them to label Wronskian polynomials. The main difference with the notions used in~\cite{Bonneux_Dunning_Stevens} is that we introduce characteristic vectors to describe $p$-cores which makes the statement of the behavior for the associated polynomials easier. The asymptotic result is given in Section~\ref{sec:MainResult} and its proof in Section~\ref{sec:Proof}.

\section{Preliminaries}\label{sec:Preliminaries}
We introduce partitions and their associated $p$-cores and $p$-quotients, as well as characteristic vectors which are in 1-1 correspondence with $p$-cores. These combinatorial concepts are well-studied and all the described material in Section~\ref{sec:Partitions} is retrieved from~\cite{Brunat_Nath}. We only explain the necessary concepts to state our results whereas extra material can be found in~\cite{Brunat_Nath,Garvan_Kim_Stanton,James_Kerber}. In Section~\ref{sec:WP} we define Wronskian polynomials which are labeled by partitions and can be factorized via $p$-cores and $p$-quotients as described in~\cite{Bonneux_Dunning_Stevens}.

\subsection{\texorpdfstring{Partitions and their $p$-cores and $p$-quotients}{Partitions and their p-cores and p-quotients}}\label{sec:Partitions}
A \emph{partition} $\lambda=(\lambda_1,\lambda_2,\dots,\lambda_r)$ is a sequence of integers such that $\lambda_1\geq \lambda_2\geq \dots \geq \lambda_r>0$. The number~$r$ is called the \emph{length} of the partition and the \emph{size} is denoted by~$|\lambda|=\sum_i \lambda_i$. We associate each partition~$\lambda$ with its \emph{degree vector} $n_\lambda=(n_1,n_2,\dots,n_r)$ defined by~$n_i=\lambda_i+r-i$ for~$i=1,2,\dots,r$, and therefore $n_1>n_2>\dots>n_r>0$. We indicate the degree vector in the associated \emph{Maya diagram} $M_\lambda$ which is defined as
\begin{equation}\label{eq:MayaDiagram}
	M_\lambda
		= \{ n \in \mathbb{Z} \mid n<0\} \cup \{n_i \mid 1\leq i \leq r\} \subset \mathbb{Z}.
\end{equation}
This diagram can be visualized by a doubly-infinite sequence of consecutive boxes that are either empty or are filled with a bullet. The boxes are labeled by the integers and the~$n^\textrm{th}$ box is filled precisely when $n\in M_\lambda$. Furthermore, a vertical line is placed between the boxes labeled with~$-1$ and~$0$; subsequently we can omit the labels. We may shift the vertical line such that the sequence of filled and empty boxes remains unchanged, but the labeling differs. We call such Maya diagrams equivalent to~$M_{\lambda}$ and denote them by~$M_{\lambda}+t$ where the integer~$t$ indicates the shift, that is, $m\in M_{\lambda}$ if and only if $m+t\in M_{\lambda}+t$. There is a unique shift such that the number of empty boxes to the left of the vertical line equals the number of filled boxes to the right of it. We write~$\widehat{M}_{\lambda}$ to denote this specific Maya diagram. In mathematical physics, see for example~\cite{Date_Jimbo_Miwa}, particles represent filled boxes while empty boxes are called holes. Moreover, the number of filled boxes to the right minus the number of empty boxes to the left of the vertical line is usually called the \emph{charge} of the diagram, and so $\widehat{M}_{\lambda}$ has charge zero. Using this definition, we can create a bijection which maps Maya diagrams to their charge and associated partition so that equivalent Maya diagrams are mapped to the same partition but different charges.

\begin{example}\label{ex:EquivalentMayaDiagrams}
Let~$\lambda=(8,8,6,6,2,2,1)$ so that~$n_\lambda=(14,13,10,9,4,3,1)$. The integers in the latter sequence are indicated to the right of the vertical line in the first Maya diagram of Figure~\ref{fig:EquivalentMayaDiagrams}. If we move the vertical line 7~steps to the right, then there are 4~empty boxes to the left of it and 4~filled boxes to the right. So $\widehat{M}_{\lambda}=M_{\lambda}-7$. The Maya diagram $\widehat{M}_{\lambda}$ reveals the Young diagram (in Russian style) of the partition as shown in Figure~\ref{fig:RussianYoungDiagram}. Each filled dot corresponds to a downwards step whereas the empty ones give rise to an upwards step.
	
\begin{figure}[t]
	\centering
	\begin{tikzpicture}[scale=0.5]
		\draw[very thin,color=gray] (-4.5,1) grid (18.5,0);
		\draw[thick,color=black] (0,1.4) -- (0,-0.4);
		\draw (-5,0.5) node {$\dots$};
		\draw (19,0.5) node {$\dots$};
		\draw (-7,0.5) node {$M_{\lambda}$};
		\foreach \x in {-4,-3,-2,-1,1,3,4,9,10,13,14}
		{
			\draw (\x+0.5,0.5) node {$\bullet$};
		}
	\end{tikzpicture}

	\begin{tikzpicture}[scale=0.5]
		\draw[very thin,color=gray] (-4.5,1) grid (18.5,0);
		\draw[thick,color=black] (7,1.4) -- (7,-0.4);
		\draw (-5,0.5) node {$\dots$};
		\draw (19,0.5) node {$\dots$};
		\draw (-7,0.5) node {$\widehat{M}_{\lambda}$};
		\foreach \x in {-4,-3,-2,-1,1,3,4,9,10,13,14}
		{
			\draw (\x+0.5,0.5) node {$\bullet$};
		}
	\end{tikzpicture}
	\caption{Two equivalent Maya diagrams (Example~\ref{ex:EquivalentMayaDiagrams}).}
	\label{fig:EquivalentMayaDiagrams}
\end{figure}
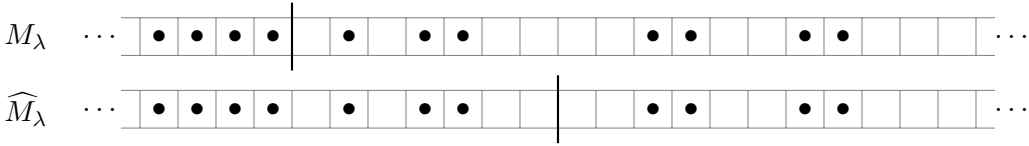
\end{example}

\begin{figure}[t]
	\centering
	\includegraphics[height=4cm]{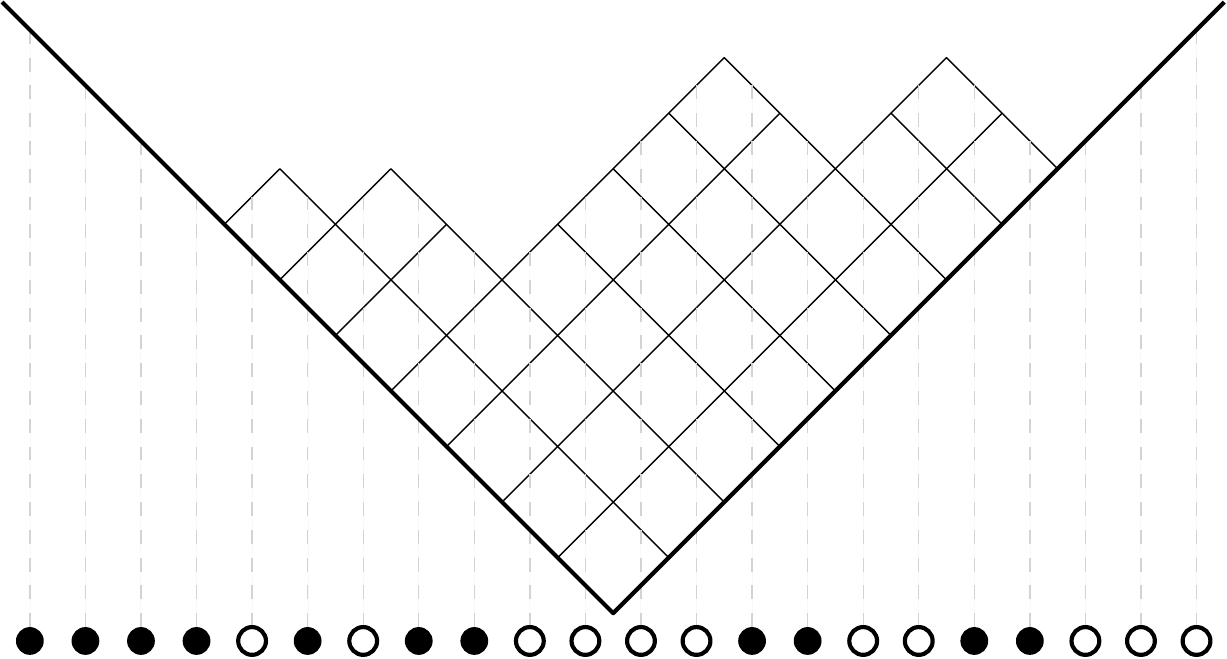}
	\caption{Young diagram in Russian style (Example~\ref{ex:EquivalentMayaDiagrams}).}
	\label{fig:RussianYoungDiagram}
\end{figure}

With each partition~$\lambda$ we can associate its \emph{$p$-core} and its \emph{$p$-quotient} for any positive integer~$p$. The $p$-quotient is a $p$-tuple of partitions denoted by~$\mu=(\mu^{(0)},\mu^{(1)},\dots,\mu^{(p-1)})$. It is defined via the $p$-modular decomposition of the degree vector~$n_\lambda$. The $p$-core is the partition~$\bar{\lambda}$ obtained by removing as many border strips (also called rim hooks) of size~$p$ as possible. Here, we describe the $p$-core via the associated \emph{characteristic vector} $c_{\lambda}=(c_0,c_1,\dots,c_{p-1})$ which is a $p$-tuple of integers such that $\sum_i c_i=0$. We now give the precise definitions.

\subsubsection{\texorpdfstring{From a partition to its characteristic vector and its $p$-quotient and vice versa}{From a partition to its characteristic vector and its p-quotient and vice versa}}

Fix a positive integer~$p$ as well as a partition~$\lambda$ and consider the Maya diagram $\widehat{M}_\lambda$. Define the Maya diagrams $M^{(0)}, M^{(1)},\dots,M^{(p-1)}$ as the $p$-modular decomposition of~$\widehat{M}_\lambda$, that is
\begin{equation}\label{eq:ModularDecomposition}
	M^{(i)}
		=\{ m \in \mathbb{Z} \mid pm+i \in\widehat{M}_{\lambda} \}
\end{equation}
for~$i=0,1,\dots,p-1$. Although the Maya diagram~$\widehat{M}_\lambda$ is defined so that there are as many empty boxes to the left of the vertical line as filled boxes to the right of it, it is not granted that the same holds for each $M^{(i)}$. We therefore define the integers $c_0,c_1,\dots,c_{p-1}$ such that
\begin{equation}\label{eq:DefinitionEntriesCharacteristicVector}
	\widehat{M}^{(i)}
		= M^{(i)} - c_i
\end{equation}
for each~$i$. We define $c_{\lambda}=(c_0,c_1,\dots,c_{p-1})$ as the characteristic vector of~$\lambda$ and the described process ensures that $\sum_i c_i=0$. Furthermore, for each~$i$ set $\mu^{(i)}$ as the partition such that its Maya diagram~$M_{\mu^{(i)}}$ is equivalent to~$M^{(i)}$, and hence also to~$\widehat{M}^{(i)}$. The tuple of partitions $\mu=(\mu^{(0)},\mu^{(1)},\dots,\mu^{(p-1)})$ is called the $p$-quotient of~$\lambda$ and its size is $|\mu|=\sum_{i} |\mu^{(i)}|$.

This process to obtain the $p$-quotient and the characteristic vector from a partition can be reversed. To do so, we first construct $M^{(0)},M^{(1)},\dots,M^{(p-1)}$ from the given $p$-quotient and characteristic vector, i.e., for each~$i$ set~$M^{(i)}$ as the equivalent Maya diagram of~$M_{\mu^{(i)}}$ such that $M^{(i)} = \widehat{M}_{\mu^{(i)}} + c_i$. Next, we set~$M$ as the Maya diagram 
\begin{equation}\label{eq:DefinitionM}
	M
		= \bigcup\limits_{i=0}^{p-1} \{ pm+i \mid m\in M^{(i)} \}.
\end{equation}
Then the associated partition~$\lambda$ is the unique partition such that $M=\widehat{M}_{\lambda}$.

\begin{example}\label{ex:quotient}
Reconsider the previous example where $\lambda=(8,8,6,6,2,2,1)$ and take~$p=3$. The Maya diagrams $M^{(0)}, M^{(1)}$ and $M^{(2)}$ are visualized in Figure~\ref{fig:quotient}. We directly observe that $\mu=((1,1), (4), (2,1))$, moreover, the dotted lines display the locations such that there are as many empty boxes to the left as filled boxes to the right of it. We therefore find that the characteristic vector is given by~$c_{\lambda}=(2,-1,-1)$.
\end{example}

\begin{remark}
The $p$-modular decompositions of two equivalent Maya diagrams are the same up to a cyclic permutation: if $\widetilde{M}=M+t$ for some integer~$t$ and if $M^{(0)},M^{(1)},\dots,M^{(p-1)}$ and $\widetilde{M}^{(0)},\widetilde{M}^{(1)},\dots,\widetilde{M}^{(p-1)}$ denote the decomposed Maya diagrams defined similarly as~\eqref{eq:ModularDecomposition}, then for each~$i$ we have that~$M^{(i)}$ and~$\widetilde{M}^{(i+t)}$ are equivalent Maya diagrams, where we interpret the indexes modulo~$p$. In the construction of the $p$-quotient of~$\lambda$, we use the (unique) Maya diagram~$\widehat{M}_{\lambda}$ and thereby fix the ordering in the $p$-quotient. Any other equivalent Maya diagram would lead to the same partitions in the $p$-quotient, but possibly in a different order.
\end{remark}

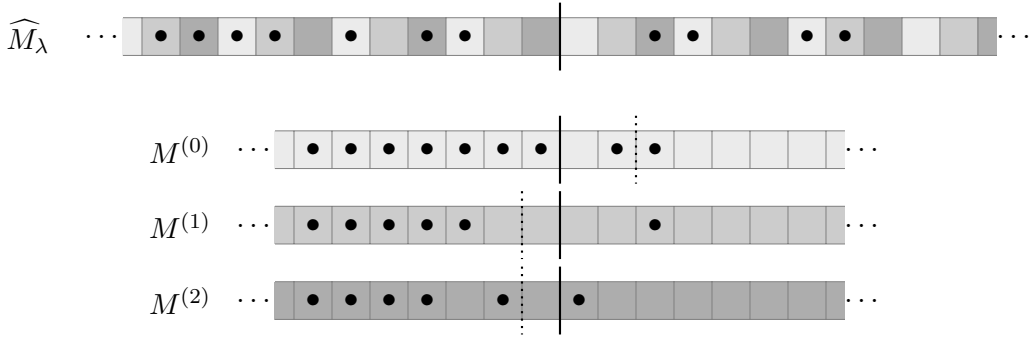
\begin{figure}[t]
	\centering
	\begin{tikzpicture}[scale=0.5]
	\foreach \x in {-9,-6,-3,0,3,6,9}
	{
		\draw[fill=black!08!white,draw=none] (7+\x,1) rectangle (7+1+\x,0);
	}
	\foreach \x in {-10,-7,-4,-1,2,5,8}
	{
		\draw[fill=black!32!white,draw=none] (7+\x,1) rectangle (7+1+\x,0);
	}
	\foreach \x in {-11,-8,-5,-2,1,4,7,10}
	{
		\draw[fill=black!20!white,draw=none] (7+\x,1) rectangle (7+1+\x,0);
	}
	\draw[fill=black!32!white,draw=none] (7+11,1) rectangle (7+0.5+11,0);
	\draw[fill=black!08!white,draw=none] (7-12+0.5,1) rectangle (7-11,0);
	\draw[very thin,color=gray] (-4.5,1) grid (18.5,0);
	\draw[thick,color=black] (7,1.4) -- (7,-0.4);
	\draw (-5,0.5) node {$\dots$};
	\draw (19,0.5) node {$\dots$};
	\draw (-7,0.5) node {$\widehat{M}_{\lambda}$};
	\foreach \x in {-4,-3,-2,-1,1,3,4,9,10,13,14}
	{
		\draw (\x+0.5,0.5) node {$\bullet$};
	}
	
	\draw[fill=black!08!white,draw=none] (-0.5,-2) rectangle (14.5,-3);
	\draw[very thin,color=gray] (-0.5,-2) grid (14.5,-3);
	\draw[thick,color=black] (7,-1.6) -- (7,-3.4);
	\draw[thick,color=black,dotted] (9,-1.6) -- (9,-3.4);
	\draw (-1,-2.5) node {$\dots$};
	\draw (15,-2.5) node {$\dots$};
	\draw (-3,-2.5) node {$M^{(0)}$};
	\foreach \x in {-7,-6,-5,-4,-3,-2,-1,1,2}
	{
		\draw (\x+7+0.5,-2.5) node {$\bullet$};
	}
	
	\draw[fill=black!20!white,draw=none] (-0.5,-4) rectangle (14.5,-5);
	\draw[very thin,color=gray] (-0.5,-4) grid (14.5,-5);
	\draw[thick,color=black] (7,-3.6) -- (7,-5.4);
	\draw[thick,color=black,dotted] (6,-3.6) -- (6,-5.4);
	\draw (-1,-4.5) node {$\dots$};
	\draw (15,-4.5) node {$\dots$};
	\draw (-3,-4.5) node {$M^{(1)}$};
	\foreach \x in {-7,-6,-5,-4,-3,2}
	{
		\draw (\x+7+0.5,-4.5) node {$\bullet$};
	}
	
	\draw[fill=black!32!white,draw=none] (-0.5,-6) rectangle (14.5,-7);
	\draw[very thin,color=gray] (-0.5,-6) grid (14.5,-7);
	\draw[thick,color=black] (7,-5.6) -- (7,-7.4);
	\draw[thick,color=black,dotted] (6,-5.6) -- (6,-7.4);
	\draw (-1,-6.5) node {$\dots$};
	\draw (15,-6.5) node {$\dots$};
	\draw (-3,-6.5) node {$M^{(2)}$};
	\foreach \x in {-7,-6,-5,-4,-2,0}
	{
		\draw (\x+7+0.5,-6.5) node {$\bullet$};
	}
	\end{tikzpicture}
	\caption{Construction of the $3$-quotient and the characteristic vector (Example~\ref{ex:quotient}).}
	\label{fig:quotient}
\end{figure}

\subsubsection{\texorpdfstring{From a characteristic vector to its $p$-core and vice versa}{From a characteristic vector to its p-core and vice versa}}

Fix a positive integer~$p$ and a partition~$\lambda$. Denote its characteristic vector by~$c_{\lambda}$ and let~$\emptyset$ be the unique partition of size~$0$ so that~$M_{\emptyset}$ equals the set of negative integers. We define the Maya diagrams $\widebar{M}^{(0)},\widebar{M}^{(1)},\dots,\widebar{M}^{(p-1)}$ by
\begin{equation}\label{eq:DefinitionCoreModularDecomposition}
	\widebar{M}^{(i)}
		= M_{\emptyset} + c_i
\end{equation}
for each~$i$, i.e., they all are equivalent to~$M_{\emptyset}=\widehat{M}_{\emptyset}$ and the charge of~$\widebar{M}^{(i)}$ is~$c_i$. Next, set~$\widebar{M}$ as the Maya diagram such that 
\begin{equation}\label{eq:DefinitionMBar}
	\widebar{M}
		= \bigcup\limits_{i=0}^{p-1} \{ pm+i \mid m\in \widebar{M}^{(i)} \}.
\end{equation}
Then, the $p$-core~$\bar{\lambda}$ of~$\lambda$ is defined as the partition such that~$M_{\bar{\lambda}}$ and~$\widebar{M}$ are equivalent Maya diagrams. That is, $\widehat{M}_{\bar{\lambda}}=\widebar{M}$ because the assumption $\sum_i c_i=0$ implies that~$\widebar{M}$ has as many filled boxes to the right of the vertical line as empty boxes to the left of it.

Again we can reverse the process. For a given $p$-core~$\bar{\lambda}$, consider the Maya diagram~$\widehat{M}_{\bar{\lambda}}$ and take its $p$-modular decomposition to obtain the Maya diagrams $\widebar{M}^{(0)},\widebar{M}^{(1)},\dots,\widebar{M}^{(p-1)}$. Then, the~$i^\textrm{th}$ element in the characteristic vector is the integer~$c_i$ such that $\widebar{M}^{(i)}-c_i= M_{\emptyset}$.

\begin{example}\label{ex:core}
In our running example we have that~$c_{\lambda}=(2,-1,-1)$. In Figure~\ref{fig:core} we display the Maya diagrams $\widebar{M}^{(0)},\widebar{M}^{(1)},\widebar{M}^{(2)}$. From these Maya diagrams we construct~$\widebar{M}$ and find that the $3$-core of~$\lambda$ is given by~$\bar{\lambda}=(4,2)$.
\end{example}

Let~$\mathbb{Y}$ be the set of all partitions and denote by~$\widebar{\mathbb{Y}}_p$ the subset of~$p$-cores, i.e., $\lambda\in\widebar{\mathbb{Y}}_p$ if and only if $\lambda=\bar{\lambda}$, or equivalently, $\widebar{\mathbb{Y}}_p$ is the set of all partitions with empty $p$-quotient. The set~$\widebar{\mathbb{Y}}_2$ is given by all partitions of the form $(k,k-1,\dots,2,1)$ with~$k\in\mathbb{Z}_{\geq0}$, and the partitions in~$\widebar{\mathbb{Y}}_3$ precisely label the generalized Okamoto polynomials, see Remark~16 in~\cite{Bonneux_Dunning_Stevens}. There is a lot of ongoing research about counting the number of specific types of core partitions, see for example~\cite{Nath} and the references therein, or~\cite{Ayyer_Sinha} where the concept of characteristic vectors is used to determine the asymptotics for the number of~$p$-core partitions of size~$n$.

\begin{remark}\label{rem:SizeCore}
For any~$p$, the size of the $p$-core can directly be computed from its characteristic vector~\cite{Garvan_Kim_Stanton}: if $c_{\lambda}=(c_0,c_1,\dots,c_{p-1})$, then the $p$-core~$\bar{\lambda}$ has size
\begin{equation}\label{eq:SizeCore}
	|\bar{\lambda}|
		= \frac{p}{2} \sum_{j=0}^{p-1} c_j^2 + \sum_{j=1}^{p-1} j c_j
\end{equation}	
which is integer valued because the condition~$\sum_i c_i=0$ implies that the number of odd elements in~$c_{\lambda}$ is even, and therefore the first sum in~\eqref{eq:SizeCore} is even. A proof of~\eqref{eq:SizeCore} is given in Section~\ref{sec:SizeCore}. 
\end{remark}

\begin{figure}[t]
	\centering
	\begin{tikzpicture}[scale=0.5]
	\draw[fill=black!08!white,draw=none] (-0.5,-2) rectangle (14.5,-3);
	\draw[very thin,color=gray] (-0.5,-2) grid (14.5,-3);
	\draw[thick,color=black] (7,-1.6) -- (7,-3.4);
	\draw[thick,color=black,dotted] (9,-1.6) -- (9,-3.4);
	\draw (-1,-2.5) node {$\dots$};
	\draw (15,-2.5) node {$\dots$};
	\draw (-3,-2.5) node {$\widebar{M}^{(0)}$};
	\foreach \x in {-7,-6,-5,-4,-3,-2,-1,0,1}
	{
		\draw (\x+7+0.5,-2.5) node {$\bullet$};
	}
	
	\draw[fill=black!20!white,draw=none] (-0.5,-4) rectangle (14.5,-5);
	\draw[very thin,color=gray] (-0.5,-4) grid (14.5,-5);
	\draw[thick,color=black] (7,-3.6) -- (7,-5.4);
	\draw[thick,color=black,dotted] (6,-3.6) -- (6,-5.4);
	\draw (-1,-4.5) node {$\dots$};
	\draw (15,-4.5) node {$\dots$};
	\draw (-3,-4.5) node {$\widebar{M}^{(1)}$};
	\foreach \x in {-7,-6,-5,-4,-3,-2}
	{
		\draw (\x+7+0.5,-4.5) node {$\bullet$};
	}
	
	\draw[fill=black!32!white,draw=none] (-0.5,-6) rectangle (14.5,-7);
	\draw[very thin,color=gray] (-0.5,-6) grid (14.5,-7);
	\draw[thick,color=black] (7,-5.6) -- (7,-7.4);
	\draw[thick,color=black,dotted] (6,-5.6) -- (6,-7.4);
	\draw (-1,-6.5) node {$\dots$};
	\draw (15,-6.5) node {$\dots$};
	\draw (-3,-6.5) node {$\widebar{M}^{(2)}$};
	\foreach \x in {-7,-6,-5,-4,-3,-2}
	{
		\draw (\x+7+0.5,-6.5) node {$\bullet$};
	}
	
	\foreach \x in {-9,-6,-3,0,3,6,9}
	{
		\draw[fill=black!08!white,draw=none] (7+\x,-9) rectangle (7+1+\x,-10);
	}
	\foreach \x in {-10,-7,-4,-1,2,5,8}
	{
		\draw[fill=black!32!white,draw=none] (7+\x,-9) rectangle (7+1+\x,-10);
	}
	\foreach \x in {-11,-8,-5,-2,1,4,7,10}
	{
		\draw[fill=black!20!white,draw=none] (7+\x,-9) rectangle (7+1+\x,-10);
	}
	\draw[fill=black!32!white,draw=none] (7+11,-9) rectangle (7+0.5+11,-10);
	\draw[fill=black!08!white,draw=none] (7-12+0.5,-9) rectangle (7-11,-10);
	\draw[very thin,color=gray] (-4.5,-9) grid (18.5,-10);
	\draw[thick,color=black] (7,-8.6) -- (7,-10.4);
	\draw (-5,-9.5) node {$\dots$};
	\draw (19,-9.5) node {$\dots$};
	\draw (-7,-9.5) node {$\widebar{M}$};
	\foreach \x in {-11,-10,-9,-8,-7,-6,-5,-4,-3,0,3}
	{
		\draw (\x+0.5+7,-9.5) node {$\bullet$};
	}
	\end{tikzpicture}
	\caption{Construction of the $3$-core (Example~\ref{ex:core}).}
	\label{fig:core}
\end{figure}
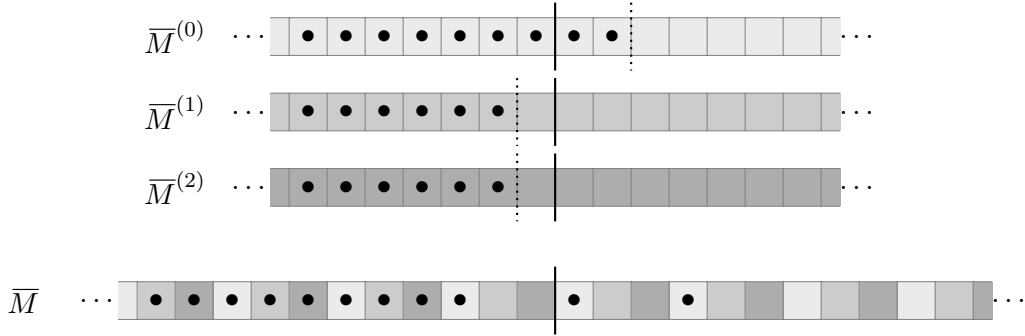

\subsubsection{Border strips and hook lengths}\label{sec:CoreViaBorderStrips}

There is another process~\cite[I.1~Ex.~8(c)]{MacDonald} to obtain the $p$-core of a partition $\lambda=(\lambda_1,\lambda_2,\dots,\lambda_r)$, namely via its \emph{Young diagram}
\begin{equation}\label{eq:YoungDiagram}
	D_\lambda
		=\{(i,j)\in\mathbb{N}^2 \mid 1\le i\le r ,\ 1\le j\le\lambda_i\}
\end{equation}
which consists of~$r$ rows, and the~$i^\textrm{th}$ row has~$\lambda_i$ boxes. The points~$(i,j)\in D_\lambda$ are depicted as unit squares with matrix-style coordinates, and so the size of the partition equals the number of boxes. To determine the $p$-core from the Young diagram, we remove strips of~$p$~connected boxes that do not contain a $2\times2$-square, such that after each removal, the remaining diagram is still associated to a partition. Such strips are called \emph{border strips} of size~$p$. At some point, we end up with a partition, possibly the empty partition, such that we cannot remove a border strip of size~$p$ any more. This partition is the $p$-core, and it is well-known that although in each step there are several possible options for deleting a border strip, the process will always end with the same partition. 

The partitions in the $p$-quotient can be read off from the Young diagram too~\cite[I.1~Ex.~8(d)]{MacDonald}. For each box in the diagram, its \emph{hook length} is defined as the number of boxes to the right and below it, hereby counting the box itself just once. If we then color all boxes whose hook lengths are a multiple of~$p$, then the Young diagrams obtained by merging all colored boxes that are in the same row or in the same column, determine the partitions in the $p$-quotient. For a clarifying example, see Figure~\ref{fig:rimhooks}. Note that this approach does not fix the ordering of the partitions.

\begin{example}\label{ex:rimhooks}
There are 4 border strips of size~$3$: two rectangles and two hooks. The left part of Figure~\ref{fig:rimhooks} visualizes a process to remove such border strips from~$\lambda=(8,8,6,6,2,2,1)$: every shade of gray matches a border strip that we remove. The remaining white boxes indicate the $3$-core, i.e., the partition~$\bar{\lambda}=(4,2)$ from which we can no longer remove any border strips of size~$3$. The right part of Figure~\ref{fig:rimhooks} indicates the hook lengths in the Young diagram and colors those that are a multiple of~$3$. We recognize the partitions $(4)$, $(1,1)$ and $(2,1)$ which are the entries of the $3$-quotient.
\end{example}

\begin{figure}[t]
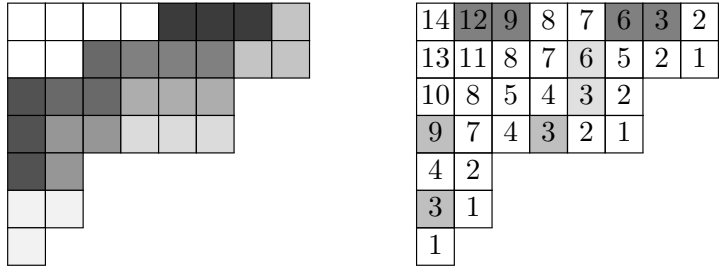

	\centering
	\ytableausetup{boxsize=1.25em}
	\ydiagram
	[*(black!05!white)]{8+0,8+0,6+0,6+0,2+0,0+2,0+1}
	*[*(black!14!white)]{8+0,8+0,6+0,3+3,2+0}
	*[*(black!23!white)]{7+1,6+2,6+0,3+0,2+0}
	*[*(black!32!white)]{7+0,6+0,3+3,3+0,2+0}
	*[*(black!41!white)]{7+0,6+0,3+0,1+2,1+1}
	*[*(black!50!white)]{7+0,3+3,3+0,1+0,1+0}
	*[*(black!59!white)]{7+0,2+1,1+2,1+0,1+0}
	*[*(black!68!white)]{7+0,2+0,0+1,0+1,0+1}
	*[*(black!77!white)]{4+3,2+0}
	*[*(white)]{4,2}
	\qquad
	\ytableausetup{boxsize=1.25em,mathmode}
	\begin{ytableau}
		14 & *(gray)12 & *(gray)9 & 8 & 7 & *(gray)6 & *(gray)3 & 2 \\
		13 & 11 & 8 & 7 & *(black!12!white)6 & 5 & 2 & 1 \\
		10 &  8 & 5 & 4 & *(black!12!white)3 & 2 \\
		*(lightgray)9 &  7 & 4 & *(lightgray)3 & 2 & 1 \\
		4 &  2 \\
		*(lightgray)3 &  1 \\
		1 
	\end{ytableau}
	\caption{The $3$-core and the partitions in the $3$-quotient obtained from the Young diagram (Example~\ref{ex:rimhooks}).}
	\label{fig:rimhooks}
\end{figure}

We end this section with the following well-known result, see for example~\cite{Garvan_Kim_Stanton,James_Kerber}. It says that each partition can be uniquely identified with its $p$-core and $p$-quotient. It immediately follows by the above explained constructions. 

\begin{lemma}\label{lem:Bijection}
Fix a positive integer~$p$. The map 
\begin{equation}\label{eq:Bijection}
	\Phi: \mathbb{Y} \to \widebar{\mathbb{Y}}_p \times \mathbb{Y}^p: \lambda \mapsto (\bar{\lambda} , \mu)
\end{equation}
where~$\bar{\lambda}$ is the $p$-core of~$\lambda$ and~$\mu$ is the $p$-quotient of~$\lambda$, is a bijection.
\end{lemma}

\subsection{Wronskian polynomials}\label{sec:WP}
Fix a positive integer~$p$ and consider the sequence of polynomials $(q_n)_{n=0}^\infty$ that have exponential generating function
\begin{equation}\label{eq:ExponentialGeneratingFunction}
	\sum_{n=0}^\infty q_n(x) \frac{t^n}{n!} 
		= \exp\left(tx - \frac{t^p}{p}\right).
\end{equation}
From this generating sequence we directly obtain that $q'_n(x)=n q_{n-1}(x)$ for all $n\geq1$ and hence $(q_n)_{n=0}^\infty$ is an Appell sequence~\cite{Bonneux_Hamaker_Stembridge_Stevens}. We can generate these polynomials via the recurrence relation
\begin{equation}\label{eq:RecurrenceRelation}
	q_n(x)
		= x q_{n-1}(x) - \frac{(n-1)!}{(n-p)!} q_{n-p}(x)
\end{equation}
for~$n\geq p$, along with the initial conditions~$q_n(x)=x^n$ for~$n<p$. Via induction, or using~\eqref{eq:ExponentialGeneratingFunction}, we then find that
\begin{equation}\label{eq:Symmetry}
	q_n(\omega_p \, x) 
		= \omega_p^n \, q_n(x)
\end{equation}
for all~$n\geq0$, where~$\omega_p=\exp(2\pi i /p)$ is the~$p^\textrm{th}$ root of unity. Using this property, among other things, one can show that all zeros lie on the $p$-star, i.e., the union of the semi-infinite intervals~$[0,\omega_p^j \, \infty)$ for~$j=0,1,\dots,p-1$. More precisely, if $n\equiv t \mod p$, then the zero at the origin of~$q_n$ has multiplicity~$t$ while the other zeros are equally divided over the semi-infinite intervals, i.e., each interval contains~$(n-t)/p$ simple zeros. These last zeros are symmetric, meaning that if~$z$ is a zero, then~$\omega_p^j \, z$ is a zero too for all~$j=0,1,\dots,p-1$.

The polynomial sequence~$(q_n)_{n=0}^\infty$ can be generalized to a net of polynomials~$(q_\lambda)_{\lambda\in\mathbb{Y}}$. For each partition~$\lambda$ we define 
\begin{equation}\label{eq:QLambda}
	q_\lambda 
		= \frac{\Wr[q_{n_1},q_{n_2},\dots,q_{n_{r}}]}{\Delta(n_{\lambda})}
		= \frac{\det \left(\frac{d^{i-1}}{dx^{i-1}} \, q_{n_j}\right)_{1\leq i ,j \leq r}}{\prod\limits_{i<j}(n_j-n_i)}
\end{equation}
where~$n_\lambda=(n_1,n_2,\dots,n_r)$ is the degree vector of~$\lambda$, and $\Delta(n_{\lambda})$ denotes the Vandermonde determinant of this vector. For~$\lambda=(n)$ we get $q_\lambda(x)= q_n(x)$ and hence the original sequence is included in the net. 

The polynomials~\eqref{eq:QLambda} were studied in several papers and it is known that the specific form of the exponential generating function in~\eqref{eq:ExponentialGeneratingFunction} ensures that~$q_\lambda$ has integer coefficients~\cite{Bonneux_Hamaker_Stembridge_Stevens}. Via the notion of~$p$-cores and $p$-quotients, where the positive integer~$p$ is obtained from the exponential generating function~\eqref{eq:ExponentialGeneratingFunction}, the following factorization is obtained in~\cite{Bonneux_Dunning_Stevens}: for any partition~$\lambda$ with $p$-core~$\bar{\lambda}$ and $p$-quotient~$\mu$, we have
\begin{equation}\label{eq:R}
	q_\lambda(x)
		= x^{|\bar{\lambda}|} R_{\lambda}(x^p)
\end{equation}
where~$R_\lambda$ is a monic polynomial of degree~$|\mu|$ with a non-vanishing constant coefficient. Several explicit expressions for the coefficients of~$R_{\lambda}$ are given in~\cite{Bonneux_Dunning_Stevens}. Our main result, given in Section~\ref{sec:MainResult}, states the asymptotic behavior of the polynomials~$R_\lambda$ when~$|\lambda|$ grows. We hereby identify~$\lambda$ with its $p$-quotient and $p$-core, and let the size of the $p$-core grow while the $p$-quotient stays fixed. This means that during this process, the degree of~$R_\lambda$ does not change while the multiplicity of the zero at the origin for~$q_\lambda$ increases to infinity. 

\begin{example}
In the trivial case~$p=1$ we have that $q_n(x)=(x-1)^n$ for any~$n\geq0$, which by directly computing of the Wronskian in~\eqref{eq:QLambda} implies that $q_\lambda(x)=(x-1)^{|\lambda|}$ for any partition~$\lambda$. The $1$-quotient equals the original partition while the $1$-core is the empty partition, hence $R_\lambda(x)=(x-1)^{|\lambda|}$. 
\end{example}

\begin{example}
When~$p=2$, we recover Wronskian Hermite polynomials which are intensively studied: there is a connection to rational solutions of the fourth Painlev\'e equation~\cite{Clarkson-zeros,Clarkson_GomezUllate_Grandati_Milson,Kajiwara_Ohta-PIV,Masoero_Roffelsen,Masoero_Roffelsen-2019,Noumi_Yamada,VanAssche} and they appear in the field of exceptional orthogonal polynomials~\cite{Duran-Recurrence,GomezUllate_Grandati_Milson-Hermite,GomezUllate_Kasman_Kuijlaars_Milson,Kuijlaars_Milson}. The results in~\cite{Clarkson_GomezUllate_Grandati_Milson} use $p$-modular decompositions of Maya diagrams to construct cyclic Maya diagrams, but they do not use the notions of~$p$-cores and $p$-quotients. Further research is needed to determine if these concepts help to interpret their results.
\end{example}

\begin{example}
The Yablonskii-Vorobiev polynomials~$(Q_n)_{n=0}^{\infty}$ can be represented via the Wronskian determinants
\begin{equation}\label{eq:YablonskiiVorobiev}
	Q_n
		= d_n \, \Wr[q_{2n-1},q_{2n-3},\dots,q_3,q_1]
\end{equation}
for all~$n\geq0$, where~$d_n$ is a normalization constant, and the sequence~$(q_n)_{n=0}^{\infty}$ has generating function $\sum_{n} q_n(x) \frac{t^n}{n!} = \exp\left(tx - \frac{4}{3}t^3\right)$, see~\cite{Kajiwara_Ohta}. Therefore, up to rescaling of the polynomials, they are contained in the net~$(q_\lambda)_{\lambda\in\mathbb{Y}}$ when~$p=3$, because the polynomial~$Q_n$ is linked to~$q_\lambda$ with $\lambda=(n,n-1,\dots,2,1)$. Note that these partitions are exactly the $2$-cores.
\end{example}

\section{Main result}\label{sec:MainResult}
For~$p=2$, it is shown in~\cite{Bonneux_Dunning_Stevens} that after a proper rescaling, the zeros of the polynomial~$R_\lambda$ are attracted to~$-1$ and~$1$ when the $2$-quotient is fixed and the size of the $2$-core grows to infinity, see~\eqref{eq:AsymptoticBehavior2}. We now generalize this result to arbitrary $p$ and offer an asymptotic result for the polynomial~$R_\lambda$ defined in~\eqref{eq:R}. We associate the $p$-core with its characteristic vector and let all entries simultaneously tend to infinity. 

\begin{theorem}\label{thm:AsymptoticResult}
Fix a positive integer~$p\geq1$ as well as a $p$-quotient $\mu=(\mu^{(0)},\mu^{(1)},\dots,\mu^{(p-1)})$. Let $c(k)=(c_0(k),c_1(k),\dots,c_{p-1}(k))\in\mathbb{Z}^{p}$ be a characteristic vector for all $k\in\mathbb{N}$ such that each entry satisfies
\begin{equation}\label{eq:ConditionVector}
	c_i(k)
		=a_i k + o(k)
\end{equation} 
as~$k\to + \infty$ for some~$a_i\in\mathbb{R}$. For any~$k\in\mathbb{N}$, denote by~$\lambda(k)$ the partition identified with the $p$-quotient~$\mu$ and the $p$-core associated to~$c(k)$. Then
\begin{equation}\label{eq:AsymptoticResult}
	\lim_{k\to + \infty} \frac{R_{\lambda(k)}\left((pk)^{p-1}x\right)}{(pk)^{(p-1)|\mu|}}
		= \prod_{i=0}^{p-1} (x-\alpha_i)^{|\mu^{(i)}|}
\end{equation}
uniformly for~$x$ in compact subsets of the complex plane, and where
\begin{equation}\label{eq:Alpha}
	\alpha_i
		= \prod_{\substack{j=0 \\ j\neq i}}^{p-1} (a_i-a_j)
\end{equation}
for~$i=0,1,\dots,p-1$.
\end{theorem}

The proof of this result is given in Section~\ref{sec:Proof}, it generalizes the ideas in~\cite{Bonneux_Dunning_Stevens} which deals with the case $p=2$. Observe that accordingly to Remark~\ref{rem:SizeCore}, the size of the $p$-core identified with the characteristic vector~$c(k)$ grows to infinity when $k\to + \infty$, and hence the size of~$\lambda(k)$ tends to infinity as well. The rescaled zeros of the polynomial~$R_{\lambda(k)}$ tend to real values and the sizes of the partitions in the $p$-quotient indicate how many zeros are attracted to each value. For~$p=1$, the asymptotic result \eqref{eq:AsymptoticResult} trivializes as the fraction in the left-hand side is independent of~$k$ and both sides trivially equal $(x-1)^{|\lambda|}$. It is also possible to determine the asymptotic behavior for~$k\to-\infty$: when~$p$ is odd it just equals the limit~\eqref{eq:AsymptoticResult}, when~$p$ is even we replace~$k$ by~$-k$ and apply the result of Theorem~\ref{thm:AsymptoticResult} whereby we substitute~$a_i$ with~$-a_i$ for all~$i$ in~\eqref{eq:ConditionVector}.

\begin{example}\label{ex:Zeros}
Take~$p=3$, fix the $3$-quotient $\mu=((2),(2,1),(1,1))$, and define the vector $c(k)=(-2k,-k,3k)$ so that the sum of the elements in this vector equals~0 for any~$k\in\mathbb{N}$. All non-zero roots of the polynomials~$q_{\lambda(k)}$ for~$k\in\{5,7,9,11\}$ are plotted in the right part of Figure~\ref{fig:ConvergenceZeros}. The left part shows the rescaled zeros of~$R_{\lambda(k)}$ as in Theorem~\ref{thm:AsymptoticResult}. Although the convergence is very slow, we observe that the zeros on the left are attracted by~$-4$, $5$ and~$20$. 
\end{example}

\begin{figure}[t]
	\centering
	\includegraphics[height=7cm,trim=60 340 160 60,clip=true]{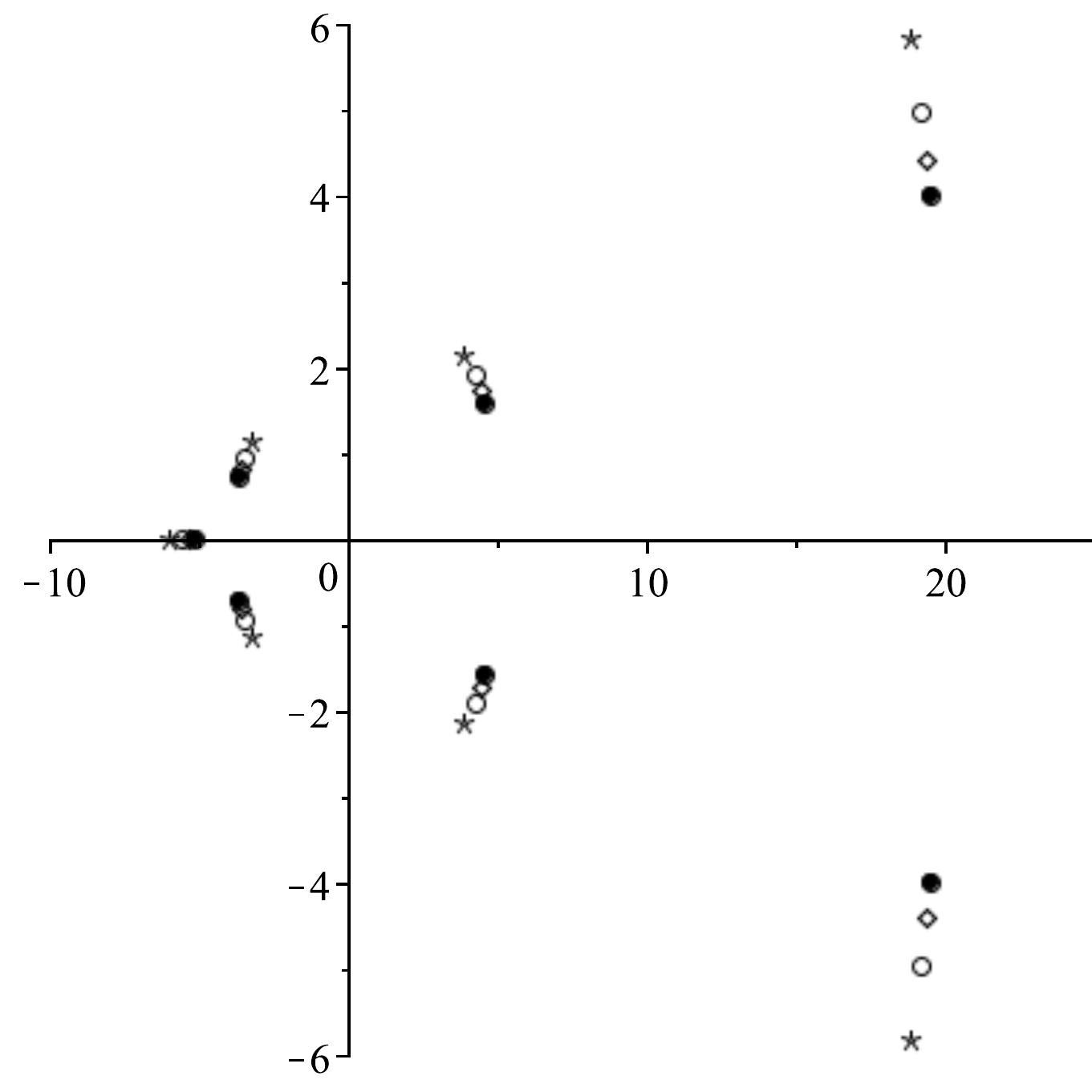}
	\qquad
	\qquad
	\includegraphics[height=7cm,trim=60 340 160 60,clip=true]{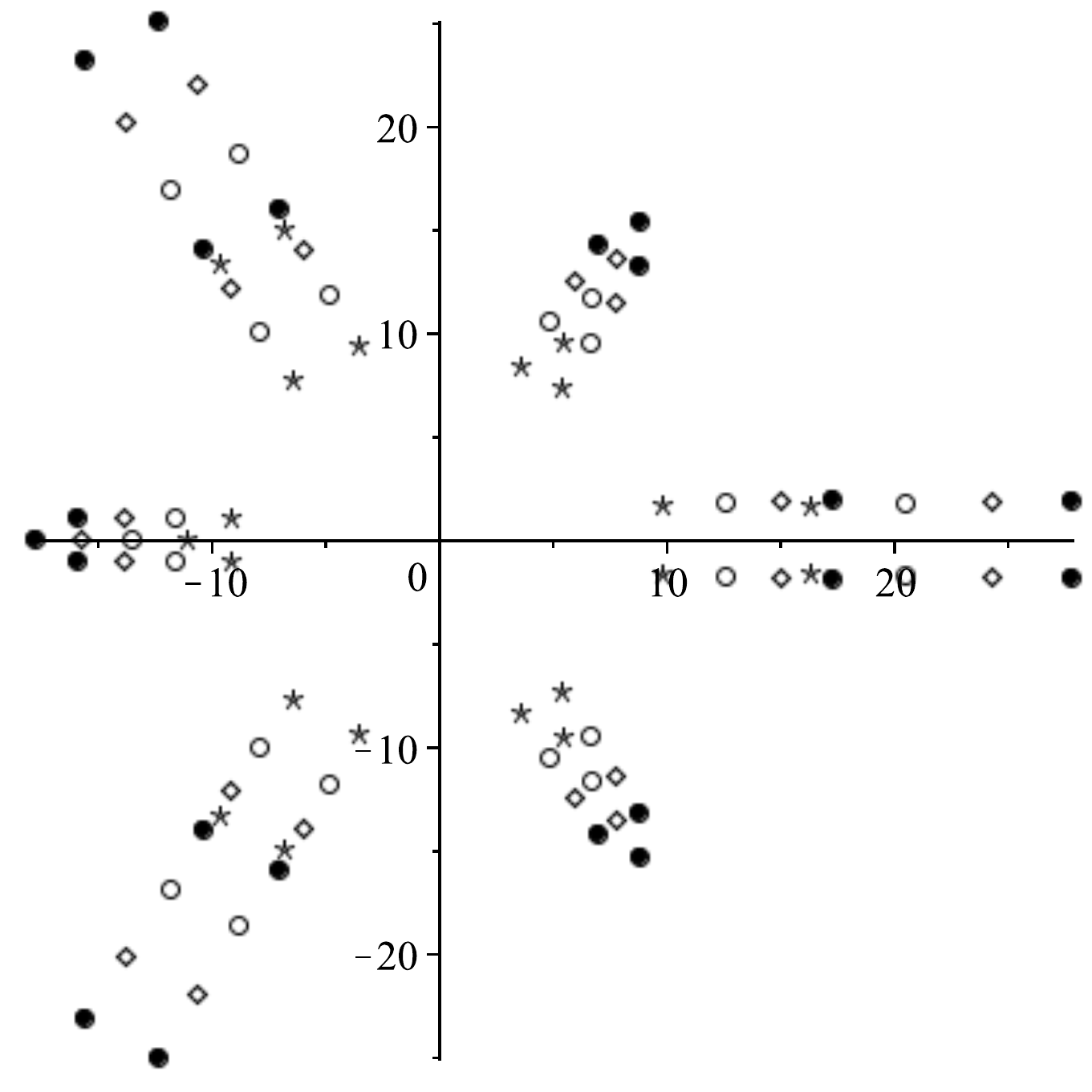}
	
	\tiny
	\begin{tabular}{| c c c c c |}
		\hline
		$k$ & $5$ & $7$ & $9$ & $11$ \\
		& $\star$ & $\circ$ & $\diamond$ & $\bullet$ \\
		\hline
	\end{tabular}
	\caption{The zeros of the polynomials~$q_{\lambda(k)}$ and~$R_{\lambda(k)}$ (Example~\ref{ex:Zeros}).}
	\label{fig:ConvergenceZeros}
\end{figure}

\begin{example}
Take~$p=2$, $a\in\mathbb{Z}$ and set~$c(k)=(-ak,ak)$. Then the asymptotic result in~\eqref{eq:AsymptoticResult} states that
\begin{equation}\label{eq:Remark2}
	\lim_{k\to + \infty} \frac{R_{\lambda(k)}\left(2akx\right)}{(2ak)^{|\mu|}}
		= (x+2a)^{|\mu^{(0)}|}(x-2a)^{|\mu^{(1)}|}.
\end{equation}	
When~$a\geq0$, the characteristic vector~$c(k)$ is associated with the $2$-core $(2ak,2ak-1,\dots,2,1)$. Similarly, when~$a<0$, then~$c(k)$ corresponds to the $2$-core $(2ak-1,2ak-2,\dots,2,1)$. Hence, if we rescale~\eqref{eq:Remark2} such that the zeros are attracted by~$-1$ and~$1$, we recover the asymptotic result in~\cite{Bonneux_Dunning_Stevens} as described in~\eqref{eq:AsymptoticBehavior2}.
\end{example}

\begin{remark}
In~\cite{Bonneux_Dunning_Stevens}, where the case~$p=2$ is studied, it is mentioned that the location of the zeros of~$q_\lambda$ reveals the $2$-quotient of~$\lambda$, that is, from Figure~9 in~\cite{Bonneux_Dunning_Stevens} we recognize the associated $2$-quotient. For arbitrary~$p$, a similar observation can be made: we somehow can determine the $3$-quotient from the exact location of the zeros in the plots of Figure~\ref{fig:ConvergenceZeros}, but more research is needed to obtain precise statements.
\end{remark}

\section{Proof of the main result}\label{sec:Proof}
In this section we prove Theorem~\ref{thm:AsymptoticResult}. We first explain the general idea and then provide all the details in the upcoming subsections. The structure of the proof is similar to the proof in~\cite{Bonneux_Dunning_Stevens} which covers the case~$p=2$, however, the concept of characteristic vectors, which is not used in~\cite{Bonneux_Dunning_Stevens}, simplifies several aspects. Throughout this section, let~$p$ be a fixed positive integer.

\paragraph{Step 0: setting the scene.}
Fix some $p$-quotient $\mu=(\mu^{(0)},\mu^{(1)},\dots,\mu^{(p-1)})\in\mathbb{Y}^{p}$. Let $c(k)=(c_0(k),c_1(k),\dots,c_{p-1}(k))\in\mathbb{Z}^{p}$ be a vector such that $\sum_i c_i(k)=0$ for any~$k\in\mathbb{N}$ and assume that the asymptotic behavior for each entry is given by~\eqref{eq:ConditionVector}. For any~$k\in\mathbb{N}$, denote~$\lambda(k)$ as the unique partition identified with the $p$-quotient~$\mu$ and the $p$-core associated to~$c(k)$, see Lemma~\ref{lem:Bijection}. For each~$k$, the polynomial~$R_{\lambda(k)}$ defined in~\eqref{eq:R} has degree~$|\mu|$, so we can write
\begin{equation}\label{eq:ExpansionR}
	R_{\lambda(k)}(x)
		= \sum_{j=0}^{|\mu|} r_{\lambda(k),j} \, x^{|\mu|-j}.
\end{equation}
Several expressions for the above coefficients are given in~\cite[Section~7]{Bonneux_Dunning_Stevens}, we will use the expression in terms of hook lengths. That is,
\begin{equation}\label{eq:Coefficient}
	r_{\lambda(k),j}
		= (-1)^j \binom{|\mu|}{j} 
		\sum_{\tilde{\mu}<_j \mu} (-1)^{\htt_p(\lambda/\tilde{\lambda})}
		\frac{F^{(p)}_{\tilde{\mu}} \,F^{(p)}_{\mu/\tilde{\mu}}}{F^{(p)}_{\mu}} 
		\,
		\frac{H_{\text{non-$p$-fold}}(\lambda(k))}{H_{\text{non-$p$-fold}}(\tilde{\lambda}(k))}
\end{equation}
for all~$j=0,1\dots,|\mu|$, and where the numerator and the denominator in the last fraction are products over the hook lengths in the partition that are not a multiple of~$p$, see Section~\ref{sec:HookLengths}. Note that this fraction is the only term in the expression that depends on~$k$. The other combinatorial concepts as well as all of the notation used in~\eqref{eq:Coefficient} are explained in Section~\ref{sec:CoresAndQuotients}.

\paragraph{Step 1: expressing the hook lengths.} 
In the first step we determine the asymptotic behavior as~$k$ tends to infinity for the terms in the last fraction in~\eqref{eq:Coefficient}, see Lemma~\ref{lem:HookLengths}. This is the most technical part of the proof, all details are clustered in Section~\ref{sec:HookLengths}.

\paragraph{Step 2: the behavior of the coefficients.}
We apply Lemma~\ref{lem:HookLengths} to rewrite~\eqref{eq:Coefficient} and show that
\begin{equation}\label{eq:CoefficientLeadingBehavior}
	(pk)^{-(p-1)j} r_{\lambda(k),j}
		= L_j + o(1)
\end{equation}
as~$k\to + \infty$, and where 
\begin{equation}\label{eq:L}
	L_j
		=
		(-1)^{j}
		\sum_{l_0+\dots+l_{p-1}=j} 
		\binom{|\mu^{(0)}|}{l_0}\binom{|\mu^{(1)}|}{l_1} \dots \binom{|\mu^{(p-1)}|}{l_{p-1}} \,
		\alpha_0^{l_0} \,\alpha_1^{l_1} \cdots \alpha_{p-1}^{l_{p-1}}	
\end{equation}
with~$\alpha_i$ defined in~\eqref{eq:Alpha}. In the above expression, we sum over all $p$-tuples of non-negative integers $l_0,l_1,\dots,l_{p-1}$ such that $\sum_i l_i=j$. Note that~$L_j$ possibly vanishes. The details are given in Section~\ref{sec:BehaviorCoefficients}.

\paragraph{Step 3: deriving the asymptotic behavior.}
As a final step, we compute the limit~\eqref{eq:AsymptoticResult} in Section~\ref{sec:AsymptoticBehavior} using~\eqref{eq:ExpansionR} and~\eqref{eq:CoefficientLeadingBehavior}. A crucial ingredient to obtain the asymptotic behavior is the binomial expansion
\begin{equation}\label{eq:BinomialExpansion}
	\prod_{i=0}^{p-1} (x-\alpha_i)^{|\mu^{(i)}|}
		= \sum_{j=0}^{|\mu|} L_j \, x^{|\mu|-j}
\end{equation}
with~$L_j$ defined in~\eqref{eq:L}. 

\subsection{The Young lattice}\label{sec:CoresAndQuotients}
The set of all partitions~$\mathbb{Y}$ forms a lattice, called the \emph{Young lattice}, partially ordered by inclusion of Young diagrams, see~\eqref{eq:YoungDiagram}. That is, $\mu\leq\lambda$ if $\mu_i \leq \lambda_i$ for all possible~$i$. For any non-negative integer~$j$ we write $\mu<_j\lambda$ if $\mu \leq \lambda$ and $|\lambda|-|\mu|=j$. If~$\mu<_j\lambda$, then a path from~$\mu$ to~$\lambda$ is a finite sequence of partitions $(\gamma^{0},\gamma^{1},\dots,\gamma^{j})$ such that $\mu =\gamma^{0}<_1\gamma^{1}<_1\dots<_1\gamma^{j-1}<_1\gamma^{j}=\lambda$.

For any two partitions~$\mu$ and~$\lambda$ with~$\mu\leq\lambda$, the difference of the Young diagrams denoted by~$D_{\lambda/\mu} = D_\lambda \setminus D_\mu$ is called a \emph{skew diagram} of shape~$\lambda/\mu$. A \emph{border strip} is a connected skew diagram~$\lambda/\mu$ that does not contain any $2 \times 2$ squares. Its \emph{height}, denoted $\htt(\lambda/\mu)$, is one less than the number of rows it occupies. If~$\mu$ is obtained from~$\lambda$ by removing several border strips of size~$p$ consecutively, then $\htt_p(\lambda/\mu)$ denotes the sum of all heights. Although~$\mu$ probably can be obtained in multiple ways from~$\lambda$, the parity of the number~$\htt_p(\lambda/\mu)$ is always the same.

The unique partition of size~$0$ is~$\emptyset$, and the integer~$F_{\lambda}$ denotes the number of paths in the Young lattice from the empty partition~$\emptyset$ to~$\lambda$. Similarly, the number~$F_{\lambda/\mu}$ is the number of paths from~$\mu$ to~$\lambda$. Formulated differently, $F_{\lambda/\mu}$ denotes the number of \textit{standard Young tableaux} of the skew diagram~$\lambda/\mu$. Subsequently, for any partition~$\lambda$ and any non-negative integer~$j$, we have
\begin{equation}\label{eq:IdentityCountingPaths}
	\sum_{\mu <_j \lambda} F_{\mu} \,F_{\lambda/\mu}
		= F_{\lambda}
\end{equation}
where we sum over all partitions~$\mu$ such that $\mu <_j \lambda$. Observe that both sides count the number of paths from~$\emptyset$ to~$\lambda$, i.e., each summand in the left-hand side counts the number of paths that passes through~$\mu$. 

The set of all $p$-quotients forms the product lattice~$\mathbb{Y}^p$. For any $\mu,\tilde{\mu}\in\mathbb{Y}^p$ we say that $\tilde{\mu}\leq \mu$ if $\tilde{\mu}^{(i)} \leq \mu^{(i)}$ for all $i$. The \textit{size} is defined as $|\mu|=\sum_{i} |\mu^{(i)}|$ and for any non-negative integer~$j$ we write $\tilde{\mu} <_j \mu$ if $\tilde{\mu} \leq \mu$ and $\lvert \tilde{\mu} \rvert +j =\lvert \mu \rvert$. A path from~$\tilde{\mu}$ to~$\mu$ is a finite sequence of $p$-tuples of partitions $(\gamma^{0},\gamma^{1},\dots,\gamma^{j})$ such that $\tilde{\mu} =\gamma^{0}<_1\gamma^{1}<_1\dots<_1\gamma^{j-1}<_1\gamma^{j}=\mu$. The number of lattice paths from~$\tilde{\mu}$ to~$\mu$ is denoted by~$F_{\mu/\tilde{\mu}}^{(p)}$ and obviously equals
\begin{equation}\label{eq:Fp}
	F_{\mu/\tilde{\mu}}^{(p)} 
		= \binom{\lvert \mu \rvert - \lvert \tilde{\mu} \rvert}{\lvert \mu^{(0)} \rvert - \lvert \tilde{\mu}^{(0)} \rvert, \lvert \mu^{(1)} \rvert - \lvert \tilde{\mu}^{(1)} \rvert, \dots, \lvert \mu^{(p-1)} \rvert - \lvert \tilde{\mu}^{(p-1)} \rvert} \prod_{i=0}^{p-1} F_{\mu^{(i)}/\tilde{\mu}^{(i)}}.
\end{equation}
We write~$F_{\mu}^{(p)}$ instead of~$F_{\mu/\emptyset}^{(p)}$ when $\tilde{\mu}=\emptyset$ is the tuple of empty partitions. 

An essential link between the notion of border strips and the $p$-quotients is that for any two partitions~$\lambda$ and~$\tilde{\lambda}$ we have that $\lambda/\tilde{\lambda}$ is a border strip of size~$p$ if and only if the $p$-quotients~$\tilde{\mu}$ and~$\mu$ satisfy  $\tilde{\mu} <_1 \mu$. Hence, the process of removing border strips from a partition to obtain the $p$-core can be described via the $p$-quotient: for each removal the size of the $p$-quotient decreases by 1. This then implies that
\begin{equation}\label{eq:SizeLambda}
	|\lambda| 
		= |\bar{\lambda}|+ p \, |\mu|
\end{equation}
where~$\bar{\lambda}$ denotes the $p$-core of~$\lambda$. More details can be found in~\cite{Brunat_Nath,Garvan_Kim_Stanton,MacDonald}. We end this section with the following result which we need in the proof of Lemma~\ref{lem:NumberOfFactors}. 

\begin{lemma}\label{lem:HeightMayaDiagram}
Take two partitions~$\tilde{\lambda}$ and~$\lambda$ such that~$\lambda/\tilde{\lambda}$ is a border strip of size~$p$. Then the Maya diagrams~$\widehat{M}_\lambda$ and~$\widehat{M}_{\tilde{\lambda}}$ differ in two boxes: $\widehat{M}_\lambda$ is obtained from~$\widehat{M}_{\tilde{\lambda}}$ by moving a bullet~$p$ steps to the right hereby passing $\htt(\lambda/\tilde{\lambda})$ other bullets.
\end{lemma}
\begin{proof}
The partitions~$\tilde{\lambda}$ and~$\lambda$ have the same $p$-core as they differ by a border strip of size~$p$, moreover, we have that $\tilde{\mu}<_1 \mu$ where~$\tilde{\mu}$ (respectively~$\mu$) denotes the $p$-quotient of~$\tilde{\lambda}$ (respectively~$\lambda$). Hence, if we construct~$\widehat{M}_\lambda$ and~$\widehat{M}_{\tilde{\lambda}}$ from their $p$-core and $p$-quotient, we get that they only differ in two boxes: $\widehat{M}_\lambda$ is obtained from~$\widehat{M}_{\tilde{\lambda}}$ by moving a bullet~$p$ steps to the right. Let~$d$ be the number of bullets that it passes by this movement. We now argue that $\htt(\lambda/\tilde{\lambda})=d$.
	
Assume first that the lengths of~$\tilde{\lambda}$ and~$\lambda$ are equal. As only 1 bullet is moved in the associated Maya diagrams, the degree vectors~$n_{\tilde{\lambda}}$ and~$n_\lambda$ considered as sets only differ by 1 element, but the entries in the vectors themselves differ at~$d+1$ places. Hence, $d+1$~entries of the partitions~$\lambda$ and~$\tilde{\lambda}$ are different, so the border strip $\lambda/\tilde{\lambda}$ occupies $d+1$~rows. Accordingly to the definition, its height therefore equals~$d$. If the length of~$\tilde{\lambda}$ is smaller than the length of~$\lambda$, then a similar but slightly more technical argument can be given. This ends the proof.
\end{proof}

\begin{figure}[t]
	\centering
	\begin{tikzpicture}[scale=0.5]
	\draw (-10,-0.75) node 
	{
		\ytableausetup{boxsize=1.25em}
		\ydiagram
		[*(white)]{4,4,4,1,1}
		*[*(gray)]{4+2,4+1,4+0,1+0,1+0}
	};
	
	\draw[very thin,color=gray] (-2.5,1) grid (13.5,0);
	\draw[thick,color=black] (6,1.4) -- (6,-0.4);
	\draw (-4.5,0.5) node {$\widehat{M}_{\tilde{\lambda}}$};
	\draw (-3,0.5) node {$\dots$};
	\draw (14,0.5) node {$\dots$};
	\foreach \x in {0,1,2,4,5,9,11}
	{
		\draw (\x+0.5-2,0.5) node {$\bullet$};
	}
	\draw[gray] (10+0.5-2,0.5) node {$\bullet$};
	
	\draw[very thin,color=gray] (-2.5,-2) grid (13.5,-3);
	\draw[thick,color=black] (6,-1.6) -- (6,-3.4);
	\draw (-4.5,-2.5) node {$\widehat{M}_{\lambda}$};
	\draw (-3,-2.5) node {$\dots$};
	\draw (14,-2.5) node {$\dots$};
	\foreach \x in {0,1,2,4,5,9,11}
	{
		\draw (\x+0.5-2,-2.5) node {$\bullet$};
	}
	\draw[gray] (13+0.5-2,-2.5) node {$\bullet$};
	
	\end{tikzpicture}
	\caption{Two partitions that differ by a border strip (Example~\ref{ex:AddRimHook}).}
	\label{fig:AddRimHook}
\end{figure}
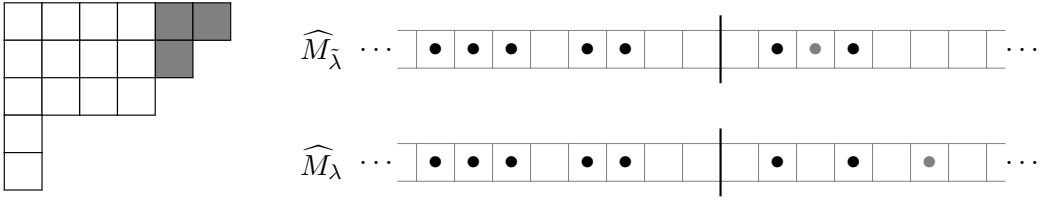

\begin{example}\label{ex:AddRimHook}
We illustrate Lemma~\ref{lem:HeightMayaDiagram} in Figure~\ref{fig:AddRimHook}. On the left we show the Young diagram of the partition $\lambda=(6,5,4,1,1)$ where we shaded a border strip of size~$3$. Removing this border strip yields the partition $\tilde{\lambda}=(4,4,4,1,1)$. For both partitions, the associated Maya diagrams are shown on the right. We clearly see that both diagrams only differ at 2 places, and that~$\widehat{M}_{\lambda}$ can be obtained from~$\widehat{M}_{\tilde{\lambda}}$ by replacing the gray bullet 3 steps to the right. It hereby passes 1 other bullet which equals the height of the removed border strip.
\end{example}

Section~\ref{sec:CoreViaBorderStrips} states that the $p$-core of a partition~$\lambda$ can be obtained by removing as many border strips of size~$p$ as possible from the Young diagram. According to Lemma~\ref{lem:HeightMayaDiagram}, each removal of such a border strip is equivalent to moving a bullet $p$~steps to the left hereby ending in an empty box in the Maya diagram~$M_\lambda$. Hence, another way to obtain the $p$-core of~$\lambda$ is by moving as many bullets as possible $p$~steps to the left. For each move the bullet ends up in an empty box, and it is allowed to move a bullet more than once. The obtained diagram equals the Maya diagram of the $p$-core.

\subsection{Hook lengths described via Maya diagrams}\label{sec:HookLengths}
For any partition~$\lambda$, the number~$H(\lambda)$ is the product of all hook lengths in the Young diagram of~$\lambda$. These hook lengths can be read off from the associated Maya diagram: each hook length is given by the distance between an empty box and a filled box to the right of it. We therefore have that
\begin{equation}\label{eq:HookLengthsMayaDiagram}
	H(\lambda)
		= \prod_{\substack{m\in M_{\lambda}, \, n\notin M_{\lambda} \\ m>n}} (m-n)
\end{equation}
where each factor represents one of the $|\lambda|$ hook lengths. Obviously, the same identity holds if we replace~$M_\lambda$ by any other equivalent Maya diagram. 

\begin{example}\label{ex:HookLengths}
Let $\lambda=(4,4,4,1,1)$. Figure~\ref{fig:HookLengths} shows the Young diagram where each box displays its hook length, and the associated Maya diagram where several hook lengths are indicated by the distances between empty and filled boxes. The other hook lengths could be visualized in a similar way.
\end{example}

\begin{figure}[t]
	\centering
	\begin{tikzpicture}[scale=0.5]
	\draw (-8,0.5) node 
	{
		\begin{ytableau}
		{\color{gray} 8}&{\color{gray} 5}&{\color{gray} 4}&{\color{gray} 3}\\
		7&4&3&2\\
		{\color{lightgray} 6}&{\color{lightgray} 3}&{\color{lightgray} 2}&{\color{lightgray} 1}\\
		2\\
		1
		\end{ytableau}
	};
	
	\draw[very thin,color=gray] (-2.5,1) grid (11.5,0);
	\draw (-3,0.5) node {$\dots$};
	\draw (12,0.5) node {$\dots$};
	\foreach \x in {0,1,2,4,5,9,10,11}
	{
		\draw (\x+0.5-2,0.5) node {$\bullet$};
	}
	
	\draw[gray] (1.5,1) to[out=90,in=90] node[pos=0.1,left]{\footnotesize{{\color{gray} 8}}} (9.5,1);
	\draw[gray] (4.5,1) to[out=90,in=90] node[pos=0.1,left]{\footnotesize{{\color{gray} 5}}} (9.5,1);
	\draw[gray] (5.5,1) to[out=90,in=90] node[pos=0.1,left]{\footnotesize{{\color{gray} 4}}} (9.5,1);
	\draw[gray] (6.5,1) to[out=90,in=90] node[pos=0.1,left]{\footnotesize{{\color{gray} 3}}} (9.5,1);
	
	\draw[lightgray] (1.5,0) to[out=-90,in=-90] node[pos=0.1,left]{\footnotesize{{\color{lightgray} 6}}} (7.5,0);
	\draw[lightgray] (4.5,0) to[out=-90,in=-90] node[pos=0.2,left]{\footnotesize{{\color{lightgray} 3}}} (7.5,0);
	\draw[lightgray] (5.5,0) to[out=-90,in=-90] node[pos=0.2,left]{\footnotesize{{\color{lightgray} 2}}} (7.5,0);
	\draw[lightgray] (6.5,0) to[out=-90,in=-90] node[pos=0.3,left]{\footnotesize{{\color{lightgray} 1}}} (7.5,0);
	\end{tikzpicture}
	\caption{Hook lengths represented via the Maya diagram (Example~\ref{ex:HookLengths}).}
	\label{fig:HookLengths}
\end{figure}
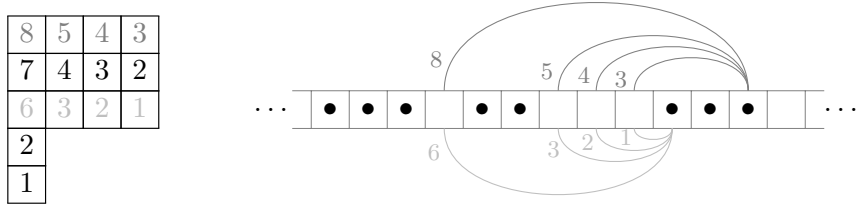

We split the set of hook lengths into two disjoint parts: the ones that are multiples of~$p$ ($p$-folds), and the others which are not multiples of~$p$ (non-$p$-folds). Accordingly we write
\begin{equation}\label{eq:Split}
	H(\lambda)
		= H_{\text{non-$p$-fold}}(\lambda) \cdot H_{\text{$p$-fold}}(\lambda).
\end{equation}
This factorization naturally appears when we express the hook lengths via filled and empty boxes in the Maya diagrams $M^{(0)},M^{(1)},\dots,M^{(p-1)}$ of the $p$-quotient. Hereby recall that if~$m$ is an element of the~$i^\textrm{th}$ Maya diagram in the quotient, then~$pm+i$ is an element of the original Maya diagram. Therefore~\eqref{eq:HookLengthsMayaDiagram} can be written as
\begin{equation}\label{eq:HookLengthsQuotient}
	H(\lambda)
		= \prod_{i=0}^{p-1} \prod_{j=0}^{p-1}\prod_{\substack{m\in M^{(i)} \\ n\notin M^{(j)} \\ pm+i>pn+j}} \big((pm+i)-(pn+j)\big)
\end{equation}
where $M^{(i)}=\widehat{M}_{\mu^{(i)}}+c_i$ for~$i=0,1,\dots,p-1$, see Section~\ref{sec:Partitions}. The hook lengths that are multiples of~$p$ are obtained when~$i=j$, this yields
\begin{align}
	H_{\text{$p$-fold}}(\lambda)
	&= \prod_{i=0}^{p-1} \prod_{\substack{m\in M^{(i)} \\ n\notin M^{(i)} \\ m>n}} p(m-n), 
	\label{eq:HookLengthsPFold}\\
	H_{\text{non-$p$-fold}}(\lambda)
	&= \prod_{i=0}^{p-1} \prod_{\substack{j=0\\ j\neq i}}^{p-1}\prod_{\substack{m\in M^{(i)} \\ n\notin M^{(j)} \\ pm+i>pn+j}} \big((pm+i)-(pn+j)\big).
	\label{eq:HookLengthsNonPFold}
\end{align}
This shows that the set of hook lengths from any $p$-core partition (a partition whose $p$-core is the partition itself) does not contain any numbers that are multiples of~$p$, moreover, the hook lengths that are a multiple of~$p$ indicate the $p$-quotient as explained in Section~\ref{sec:CoreViaBorderStrips}. Finally, using $M^{(i)}=\widehat{M}_{\mu^{(i)}}+c_i$ which says that $m \in \widehat{M}_{\mu^{(i)}}$ if and only if $m+c_i\in M^{(i)}$, we can write~\eqref{eq:HookLengthsNonPFold} as
\begin{equation}\label{eq:HookLengthsNonPFoldBis}
	H_{\text{non-$p$-fold}}(\lambda)
		= \prod_{i=0}^{p-1} \prod_{\substack{j=0\\ j\neq i}}^{p-1}\prod_{\substack{m\in \widehat{M}_{\mu^{(i)}} \\ n\notin \widehat{M}_{\mu^{(j)}} \\ p(m+c_i)+i>p(n+c_j)+j}} \big((p(m+c_i)+i)-(p(n+c_j)+j)\big).	
\end{equation}

\begin{example}\label{ex:HookLengthsQuotient}
Reconsider $\lambda=(4,4,4,1,1)$ and take~$p=3$. Its $3$-quotient is $\mu=((1),(2),(1))$ and its characteristic vector is given by~$c_\lambda=(1,-1,0)$, or equivalently, the $3$-core is $\bar{\lambda}=(1,1)$. In Figure~\ref{fig:HookLengthsQuotient} we visualize the process to obtain these data and indicate the same hook lengths as shown in the Maya diagram of Figure~\ref{fig:HookLengths}. The hook lengths that are divisible by~$3$ are obtained via pairs of empty and filled boxes in the same Maya diagram.
\end{example}

\begin{figure}[t]
	\centering
	\begin{tikzpicture}[scale=0.5]
	\foreach \x in {-6,-3,0,3,6}
	{
		\draw[fill=black!08!white,draw=none] (7+\x,1) rectangle (7+1+\x,0);
	}
	\foreach \x in {-7,-4,-1,2,5}
	{
		\draw[fill=black!32!white,draw=none] (7+\x,1) rectangle (7+1+\x,0);
	}
	\foreach \x in {-8,-5,-2,1,4,7}
	{
		\draw[fill=black!20!white,draw=none] (7+\x,1) rectangle (7+1+\x,0);
	}
	\draw[fill=black!32!white,draw=none] (7+8,1) rectangle (7+0.5+8,0);
	\draw[fill=black!08!white,draw=none] (7-9+0.5,1) rectangle (7-8,0);
	\draw[very thin,color=gray] (-1.5,1) grid (15.5,0);
	\draw[thick,color=black] (7,1.4) -- (7,-0.4);
	\draw (-2,0.5) node {$\dots$};
	\draw (16,0.5) node {$\dots$};
	\draw (-4,0.5) node {$\widehat{M}_{\lambda}$};
	\foreach \x in {-8,-7,-6,-4,-3,1,2,3}
	{
		\draw (\x+0.5+7,0.5) node {$\bullet$};
	}
	
	\draw[fill=black!08!white,draw=none] (2.5,-2) rectangle (11.5,-3);
	\draw[very thin,color=gray] (2.5,-2) grid (11.5,-3);
	\draw[thick,color=black] (7,-1.6) -- (7,-3.4);
	\draw[thick,color=black,dotted] (8,-1.6) -- (8,-3.4);
	\draw (2,-2.5) node {$\dots$};
	\draw (12,-2.5) node {$\dots$};
	\draw (0,-2.5) node {$M^{(0)}$};
	\foreach \x in {-4,-3,-2,-1,1}
	{
		\draw (\x+7+0.5,-2.5) node {$\bullet$};
	}
	
	\draw[fill=black!20!white,draw=none] (2.5,-5) rectangle (11.5,-6);
	\draw[very thin,color=gray] (2.5,-5) grid (11.5,-6);
	\draw[thick,color=black] (7,-4.6) -- (7,-6.4);
	\draw[thick,color=black,dotted] (6,-4.6) -- (6,-6.4);
	\draw (2,-5.5) node {$\dots$};
	\draw (12,-5.5) node {$\dots$};
	\draw (0,-5.5) node {$M^{(1)}$};
	\foreach \x in {-4,-3,0}
	{
		\draw (\x+7+0.5,-5.5) node {$\bullet$};
	}
	
	\draw[fill=black!32!white,draw=none] (2.5,-8) rectangle (11.5,-9);
	\draw[very thin,color=gray] (2.5,-8) grid (11.5,-9);
	\draw[thick,color=black] (7,-7.6) -- (7,-9.4);
	\draw (2,-8.5) node {$\dots$};
	\draw (12,-8.5) node {$\dots$};
	\draw (0,-8.5) node {$M^{(2)}$};
	\foreach \x in {-4,-3,-2,0}
	{
		\draw (\x+7+0.5,-8.5) node {$\bullet$};
	}
	
	\draw[gray] (8.5,-3) to[out=-90,in=90] node[pos=0.8,left]{\footnotesize{{\color{gray} 8}}} (5.5,-5);
	\draw[gray] (8.5,-3) to[out=-90,in=90] node[pos=0.8,left]{\footnotesize{{\color{gray} 5}}} (6.5,-5);
	\draw[gray] (8.5,-3) to[out=-90,in=90] node[pos=0.9,left]{\footnotesize{{\color{gray} 4}}} (6.5,-8);
	\draw[gray] (8.5,-2) to[out=90,in=90] node[pos=0.5,above]{\footnotesize{{\color{gray} 3}}} (7.5,-2);
	
	\draw[lightgray] (7.5,-6) to[out=-90,in=-90] node[pos=0.7,left]{\footnotesize{{\color{lightgray} 3}}} (6.5,-6);
	\draw[lightgray] (7.5,-6) to[out=-90,in=-90] node[pos=0.8,left]{\footnotesize{{\color{lightgray} 6}}} (5.5,-6);
	\draw[lightgray] (7.5,-6) to[out=-90,in=90] node[pos=0.6,right]{\footnotesize{{\color{lightgray} 2}}} (6.5,-8);
	\draw[lightgray] (7.5,-5) to[out=90,in=-90] node[pos=0.25,right]{\footnotesize{{\color{lightgray} 1}}} (7.5,-3);
	
	\end{tikzpicture}
	\caption{Hook lengths represented via the Maya diagrams of the $3$-quotient (Example~\ref{ex:HookLengthsQuotient}).}
	\label{fig:HookLengthsQuotient}
\end{figure}
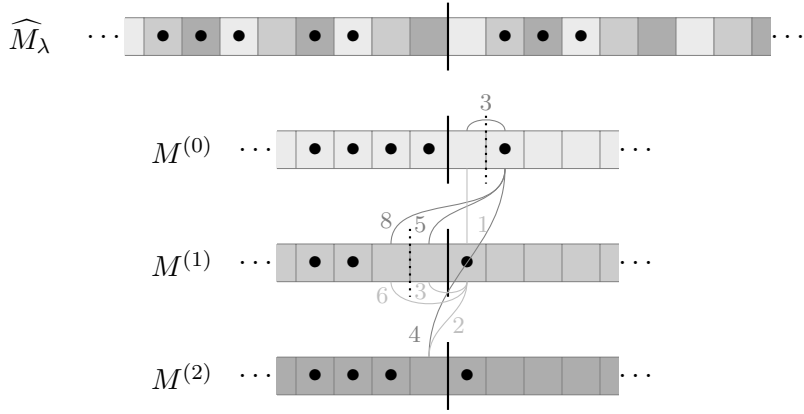

In~\cite[Section~4]{Bessenrodt}, a remarkable result about hook lengths is obtained: for any partition~$\lambda$, the set of hook lengths (including multiplicities) of the $p$-core~$\bar{\lambda}$ is a subset of the hook lengths (including multiplicities) of the partition~$\lambda$ itself. Furthermore, as all hook lengths of~$\bar{\lambda}$ are not multiples of~$p$, we have that each factor of~$H(\bar{\lambda})$ is a factor of~$H_{\text{non-$p$-fold}}(\lambda)$. The number of extra factors in~$H_{\text{non-$p$-fold}}(\lambda)$ is indicated in the following lemma.

\begin{lemma}\label{lem:NumberOfFactors}
For any positive integer~$p$, and any partition~$\lambda$ with $p$-quotient~$\mu$ and $p$-core~$\bar{\lambda}$, we have that each factor of~$H(\bar{\lambda})$ is a factor of~ $H_{\text{non-$p$-fold}}(\lambda)$, moreover, the latter product has exactly $(p-1)|\mu|$ more factors than the product~$H(\bar{\lambda})$.
\end{lemma}
\begin{proof}
The first statement follows from the reasoning stated before the lemma. For the second part, recall that the size of the partition~$\lambda$ indicates the total number of hook lengths, and that the hook lengths which are a multiple of~$p$ can be identified with the quotient~$\mu$, i.e., the number of hook lengths that are a multiple of~$p$ equals~$|\mu|$. Hence, the difference in the number of factors is given by~$|\lambda|-|\mu|-|\bar{\lambda}|$. The result then follows from~\eqref{eq:SizeLambda}.
\end{proof}

The previous lemma tells us that the fraction $H_{\text{non-$p$-fold}}(\lambda)/H(\bar{\lambda})$ is actually just a product of~$(p-1)|\mu|$ hook lengths in~$\lambda$ that are not multiples of~$p$. We now have the following result related to the set-up of Theorem~\ref{thm:AsymptoticResult}.

\begin{lemma}\label{lem:HookLengths}
Fix some positive integer~$p\geq1$ as well as a $p$-quotient $\mu=(\mu^{(0)},\mu^{(1)},\dots,\mu^{(p-1)})$. Let $c(k)=(c_0(k),c_1(k),\dots,c_{p-1}(k))\in\mathbb{Z}^{p}$ be a characteristic vector for any~$k\in\mathbb{N}$ such that each entry satisfies $c_i(k)=a_i k + o(k)$ as $k\to + \infty$ for some~$a_i\in\mathbb{R}$. For any $k\in\mathbb{N}$, denote by~$\lambda(k)$ the partition identified with the $p$-quotient~$\mu$ and the $p$-core associated to~$c(k)$. Then 
\begin{equation}\label{eq:Fraction}
	\frac{H_{\text{non-$p$-fold}}(\lambda(k))}{H(\bar{\lambda}(k))}
		= A \, k^{(p-1)|\mu|} + o(k^{(p-1)|\mu|})
\end{equation}
as~$k\to+\infty$, and where
\begin{equation}\label{eq:HookLengthsLeadingTerm}
	A = (-1)^{\htt_p(\lambda(k)/\bar{\lambda}(k))} \, p^{(p-1)|\mu|} \, \prod_{i=0}^{p-1}\alpha_{i}^{|\mu^{(i)}|}
\end{equation}
with~$\alpha_i$ defined in~\eqref{eq:Alpha}.
\end{lemma}
\begin{proof}
The first part of the proof shows the asymptotic behavior~\eqref{eq:Fraction} while the second part deals with the coefficient~\eqref{eq:HookLengthsLeadingTerm}.
	
\paragraph{Part 1.}
Lemma~\ref{lem:NumberOfFactors} states that the left-hand side of~\eqref{eq:Fraction} is an integer as each factor in the product of the denominator is also a factor of the product in the numerator, moreover, it indicates that after the division there are~$(p-1)|\mu|$ remaining factors. Using~\eqref{eq:HookLengthsNonPFoldBis}, the numerator of the fraction can be written as	
\begin{equation}\label{eq:ProofFraction0}
	H_{\text{non-$p$-fold}}(\lambda(k))
		= \prod_{i=0}^{p-1} \prod_{\substack{j=0\\ j\neq i}}^{p-1}\prod_{\substack{m\in \widehat{M}_{\mu^{(i)}} \\ n\notin \widehat{M}_{\mu^{(j)}} \\ p(m+c_i(k))+i>p(n+c_j(k))+j}} \big( p(c_i(k)-c_j(k)) +pm+i-pn-j \big).
\end{equation}
Then, using the limiting behavior of the entries in the characteristic vector, we find that
\begin{equation}\label{eq:ProofFraction1}
	H_{\text{non-$p$-fold}}(\lambda(k))
		= \prod_{i=0}^{p-1} \prod_{\substack{j=0\\ j\neq i}}^{p-1}\prod_{\substack{m\in \widehat{M}_{\mu^{(i)}} \\ n\notin \widehat{M}_{\mu^{(j)}} \\ p(m+c_i(k))+i>p(n+c_j(k))+j}} \big( pk(a_i-a_j) +o(k) \big)
\end{equation}
as~$k\to+\infty$. As in the left-hand side of~\eqref{eq:Fraction} exactly~$(p-1)|\mu|$ factors of~$H_{\text{non-$p$-fold}}(\lambda(k))$ survive, we directly get the asymptotic result via~\eqref{eq:ProofFraction1}. 
	
\paragraph{Part 2.}
We now approach by induction on~$|\mu|$ to determine the coefficient associated to the term~$k^{(p-1)|\mu|}$ in~\eqref{eq:Fraction}. If we reconsider~\eqref{eq:ProofFraction1}, we observe that the coefficient~$A$ can be determined by counting how many factors arise for any 2 different Maya diagrams, i.e., only the part~$a_i-a_j$ is relevant which depends on the associated Maya diagrams~$M^{(i)}$ and~$M^{(j)}$.
	
If~$\mu=\emptyset$, then~$\lambda(k)$ is a $p$-core for any~$k$, i.e., its $p$-core equals the partition itself. We then have $H_{\text{non-$p$-fold}}(\lambda(k))=H(\lambda(k))=H(\bar{\lambda}(k))$. Hence, the left-hand side of~\eqref{eq:Fraction} equals~1 and so the result is trivially true in this case.
	
Next fix a $p$-quotient $\mu=(\mu^{(0)},\mu^{(1)},\dots,\mu^{(p-1)})$ with~$|\mu|>0$ and assume that the result holds for any other smaller $p$-quotient. Take $\tilde{\mu}=(\tilde{\mu}^{(0)},\tilde{\mu}^{(1)},\dots,\tilde{\mu}^{(p-1)})$ such that $\tilde{\mu}<_1\mu$. That is, there exists a non-negative integer~$i^{\star}$ such that $\tilde{\mu}^{(i^{\star})} <_1 \mu^{(i^{\star})}$ while $\tilde{\mu}^{(i)} = \mu^{(i)}$ for~$i\neq i^{\star}$. Further, let~$\lambda(k)$ (respectively~$\tilde{\lambda}(k)$) be the partition identified by the $p$-quotient~$\mu$ (respectively~$\tilde{\mu}$) and $p$-core associated to~$c(k)$, and let $M^{(0)},M^{(1)},\dots,M^{(p-1)}$ (respectively $\widetilde{M}^{(0)},\widetilde{M}^{(1)},\dots,\widetilde{M}^{(p-1)}$) denote the Maya diagrams in the construction of the $p$-quotient of~$\lambda(k)$ (respectively~$\tilde{\lambda}(k)$). That is, $M^{(i)}=\widehat{M}_{\mu^{(i)}}+c_i$ and $\widetilde{M}^{(i)}=\widehat{M}_{\tilde{\mu}^{(i)}}+c_i$ for all~$i$, and note that $M^{(i)}=\widetilde{M}^{(i)}$ for~$i\neq i^{\star}$. We so obtain the following set-up.
\begin{align*}
	\lambda(k) \longleftrightarrow \big(c(k), \mu=(&\mu^{(0)},\mu^{(1)},\dots,\mu^{(p-1)}) \big) 
	\text{ with Maya diagrams } M^{(0)},M^{(1)},\dots,M^{(p-1)} 
	\\
	&\tilde{\mu}^{(i)} = \mu^{(i)} \text{ for }  i\neq i^{\star} \hspace{4cm} M^{(i)}=\widetilde{M}^{(i)} \text{ for } i\neq i^{\star}
	\\
	\tilde{\lambda}(k) \longleftrightarrow \big(c(k), \tilde{\mu}=(&\tilde{\mu}^{(0)},\tilde{\mu}^{(1)},\dots,\tilde{\mu}^{(p-1)}) \big) 
	\text{ with Maya diagrams } \widetilde{M}^{(0)},\widetilde{M}^{(1)},\dots,\widetilde{M}^{(p-1)} 	
\end{align*}
As $\tilde{\mu}^{(i^{\star})} <_1 \mu^{(i^{\star})}$, the difference between the Maya diagrams~$M^{(i^{\star})}$ and~$\widetilde{M}^{(i^{\star})}$ is given by a movement of a bullet 1 step to the right. In Figure~\ref{fig:FilledBox} we visualize this move. The crosses indicate that a box can either be filled or empty, but corresponding crosses (at the same location) in both Maya diagrams should be the same. The integer~$t$ denotes the location of the bullet in Maya diagram~$\widetilde{M}^{(i^{\star})}$ which is moved to box~$t+1$ in~$M^{(i^{\star})}$. Note that this is only possible when the box located at~$t+1$ in~$\widetilde{M}^{(i^{\star})}$ is empty.
	
We now obtain the coefficient associated to~$k^{(p-1)|\mu|}$ in the fraction $H_{\text{non-$p$-fold}}(\lambda(k))/H(\bar{\lambda}(k))$ using the induction hypothesis on $H_{\text{non-$p$-fold}}(\tilde{\lambda}(k))/H(\bar{\lambda}(k))$. As both fractions have the same denominator, the same factors in both fractions are canceled. Hence, it is sufficient to describe the difference between all factors in the numerators to compare the behavior of both fractions. 
	
To study the factors in $H_{\text{non-$p$-fold}}(\tilde{\lambda}(k))$ and in $H_{\text{non-$p$-fold}}(\lambda(k))$, we describe these factors by pairs of empty and filled boxes in the $p$-quotient as stated in~\eqref{eq:HookLengthsNonPFoldBis}. As $M^{(i)}=\widetilde{M}^{(i)}$ for~$i\neq i^{\star}$, and as the Maya diagrams~$M^{(i^{\star})}$ and~$\widetilde{M}^{(i^{\star})}$ only differ in two boxes, see Figure~\ref{fig:FilledBox}, we have that all factors in both products are the same except those related with the boxes~$t$ and~$t+1$ in~$\widetilde{M}^{(i^{\star})}$ and~$M^{(i^{\star})}$. We now precisely describe these factors.

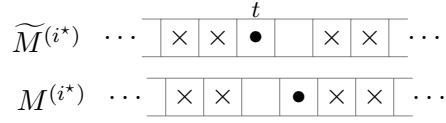
\begin{figure}[t]
	\centering
	\begin{tikzpicture}[scale=0.5]
	\draw[very thin,color=gray] (4.5,1) grid (11.5,0);
	\draw (4,0.5) node {$\dots$};
	\draw (12,0.5) node {$\dots$};
	\draw (2,0.5) node {$\widetilde{M}^{(i^{\star})}$};
	\foreach \x in {-2,-1,2,3}
	{
		\draw (\x+0.5+7,0.5) node {$\times$};
	}
	\foreach \x in {0}
	{
		\draw (\x+0.5+7,0.5) node {$\bullet$};
	}
	\draw[] (7.5,1.25) node {\footnotesize{$t$}};
	\end{tikzpicture}
	\vspace{0.125cm}
	
	\begin{tikzpicture}[scale=0.5]
	\draw[very thin,color=gray] (4.5,1) grid (11.5,0);
	\draw (4,0.5) node {$\dots$};
	\draw (12,0.5) node {$\dots$};
	\draw (2,0.5) node {$M^{(i^{\star})}$};
	\foreach \x in {-2,-1,2,3}
	{
		\draw (\x+0.5+7,0.5) node {$\times$};
	}
	\foreach \x in {1}
	{
		\draw (\x+0.5+7,0.5) node {$\bullet$};
	}
	\end{tikzpicture}
	\caption{The difference between the Maya diagrams in the $p$-quotients.}
	\label{fig:FilledBox}
\end{figure}
	
First, consider the factors in $H_{\text{non-$p$-fold}}(\tilde{\lambda}(k))$ that are obtained via the empty box located at $t+1$ in $\widetilde{M}^{(i^{\star})}$. These factors are established in combination with filled boxes located at $t+1$ in $\widetilde{M}^{(i)}$ for~$i>i^{(\star)}$, or with filled boxes located at~$n>t+1$ in~$\widetilde{M}^{(i)}$ for~$i\neq i^{\star}$. Pick any such factor which is established via a filled box in~$\widetilde{M}^{(i)}$ and recall that $\widetilde{M}^{(i)}=M^{(i)}$. Then there is a related factor in $H_{\text{non-$p$-fold}}(\lambda(k))$ which is obtained via the same filled box in~$M^{(i)}$ in combination with the empty box located at~$t$ in~$M^{(i^{\star})}$.
	
Second, consider the factors in $H_{\text{non-$p$-fold}}(\tilde{\lambda}(k))$ that are obtained via the filled box located at~$t$ in~$\widetilde{M}^{(i^{\star})}$. These factors are established in combination with empty boxes located at $t$ in $\widetilde{M}^{(i)}$ for~$i<i^{(\star)}$, or with empty boxes located at~$n<t$ in~$\widetilde{M}^{(i)}$ for~$i\neq i^{\star}$. Pick any such factor which is established via an empty box in~$\widetilde{M}^{(i)}$ and recall that $\widetilde{M}^{(i)}=M^{(i)}$. Then there is a related factor in $H_{\text{non-$p$-fold}}(\lambda(k))$ which is obtained via the same empty box in~$M^{(i)}$ in combination with the filled box located at~$t+1$ in~$M^{(i^{\star})}$.
	
As an intermediate conclusion, the two types of described factors in $H_{\text{non-$p$-fold}}(\tilde{\lambda}(k))$, obtained via filled and empty boxes in~$\widetilde{M}^{(i^{\star})}$ and~$\widetilde{M}^{(i)}$, are in 1-1 correspondence with factors in $H_{\text{non-$p$-fold}}(\lambda(k))$ which are obtained via filled and empty boxes in~$M^{(i^{\star})}$ and~$M^{(i)}$. Stated otherwise, the factors themselves are different in both products, but they are obtained via the same empty or the same filled boxes in the Maya diagrams $\widetilde{M}^{(i)}=M^{(i)}$ for~$i\neq i^{\star}$. Hence, the linear behavior in~$k$ for these factors is the same in both products~$H_{\text{non-$p$-fold}}(\tilde{\lambda}(k))$ and~$H_{\text{non-$p$-fold}}(\lambda(k))$.
	
Third, the movement of bullet as indicated in Figure~\ref{fig:FilledBox} induces exactly $p-1$ new factors in $H_{\text{non-$p$-fold}}(\lambda(k))$ that have no analogue in $H_{\text{non-$p$-fold}}(\tilde{\lambda}(k))$. We now precisely describe these new factors. 
	
Assume first that~$j<i^{\star}$ and consider~$M^{(j)}$. The box located at~$t+1$ is either filled or empty. If it is filled, then this box together with the empty box at~$t$ in~$M^{(i^{\star})}$ leads to a factor. If it is empty, then this box together with the filled box at~$t+1$ in~$M^{(i^{\star})}$ leads to a factor, see Figure~\ref{fig:NewFactors1}. These factors were not present in~$H_{\text{non-$p$-fold}}(\tilde{\lambda}(k))$ because in~$\widetilde{M}^{(i^{\star})}$, the box located at~$t$ is filled while the box at~$t+1$ is empty. 
	
\begin{figure}[b]
	\centering
	\begin{tikzpicture}[scale=0.5]
		\draw[very thin,color=gray] (4.5,1) grid (11.5,0);
		\draw (4,0.5) node {$\dots$};
		\draw (12,0.5) node {$\dots$};
		\draw (2,0.5) node {$M^{(j)}$};
		\foreach \x in {-2,-1,0,2,3}
		{
			\draw (\x+0.5+7,0.5) node {$\times$};
		}
		\foreach \x in {}
		{
			\draw (\x+0.5+7,0.5) node {$\bullet$};
		}
		\draw[] (7.5,1.25) node {\footnotesize{$t$}};
		
		\draw[very thin,color=gray] (4.5,-2) grid (11.5,-3);
		\draw (4,-2.5) node {$\dots$};
		\draw (12,-2.5) node {$\dots$};
		\draw (2,-2.5) node {$M^{(i^{\star})}$};
		\foreach \x in {-2,-1,2,3}
		{
			\draw (\x+0.5+7,-2.5) node {$\times$};
		}
		\foreach \x in {1}
		{
			\draw (\x+0.5+7,-2.5) node {$\bullet$};
		}
		
		\draw[] (8.5,-2) to[out=90,in=-90] node[pos=0.5,right]{\footnotesize{new}} (8.5,0);
		
		\draw[very thin,color=gray] (4.5+15,1) grid (11.5+15,0);
		\draw (4+15,0.5) node {$\dots$};
		\draw (12+15,0.5) node {$\dots$};
		\draw (2+15,0.5) node {$M^{(j)}$};
		\foreach \x in {-2,-1,0,2,3}
		{
			\draw (\x+0.5+7+15,0.5) node {$\times$};
		}
		\foreach \x in {1}
		{
			\draw (\x+0.5+7+15,0.5) node {$\bullet$};
		}
		\draw[] (7.5+15,1.25) node {\footnotesize{$t$}};
		
		\draw[very thin,color=gray] (4.5+15,-2) grid (11.5+15,-3);
		\draw (4+15,-2.5) node {$\dots$};
		\draw (12+15,-2.5) node {$\dots$};
		\draw (2+15,-2.5) node {$M^{(i^{\star})}$};
		\foreach \x in {-2,-1,2,3}
		{
			\draw (\x+0.5+7+15,-2.5) node {$\times$};
		}
		\foreach \x in {1}
		{
			\draw (\x+0.5+7+15,-2.5) node {$\bullet$};
		}
		
		\draw[] (7.5+15,-2) to[out=90,in=-90] node[pos=0.5,right]{\footnotesize{new}} (8.5+15,0);
	\end{tikzpicture}
	\caption{New factors for~$j<i^{\star}$.}
	\label{fig:NewFactors1}
\end{figure}
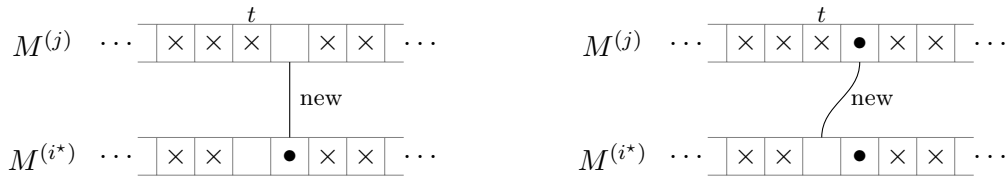 
	
Next, assume that~$j>i^{\star}$ and consider the box located at~$t$ in~$M^{(j)}$. If it is empty, then this box together with the filled box at~$t+1$ in~$M^{(i^{\star})}$ leads to a factor, if it is filled, then a factor is obtained via a combination with the empty box at~$t$ in~$M^{(i^{\star})}$, see Figure~\ref{fig:NewFactors2}. Again, these are new factors.
	
From the third part we conclude that there are exactly~$p-1$ new factors in $H_{\text{non-$p$-fold}}(\lambda(k))$ that were not present in~$H_{\text{non-$p$-fold}}(\tilde{\lambda}(k))$, and that each such factor is obtained via a filled and an empty box in~$M^{(i^{\star})}$ and~$M^{(j)}$ with~$j\neq i^{\star}$. Via~\eqref{eq:HookLengthsNonPFoldBis}, the product of these new factors is given by
\begin{equation}\label{eq:ExtraFactors}
	\prod_{\substack{j=0\\ j\neq i^{\star}}}^{p-1} \big((-1)^{h_{i^{\star},j}} \, p(c_{i^{\star}}(k)-c_j(k)) +d_{i^{\star},j} )\big)
		= (pk)^{p-1} \prod_{\substack{j=0\\ j\neq i^{\star}}}^{p-1} \big( (-1)^{h_{i^{\star},j}} (a_{i^{\star}}-a_j) \big) + o(k^{p-1})
\end{equation}
as~$k\to+\infty$ for some $d_{i^{\star},j}\in\mathbb{Z}$, and where $(-1)^{h_{i^{\star},j}}$ is needed to indicate if the factor is obtained from a filled box in~$M^{(i^{\star})}$ or from a filled box in~$M^{(j)}$. 

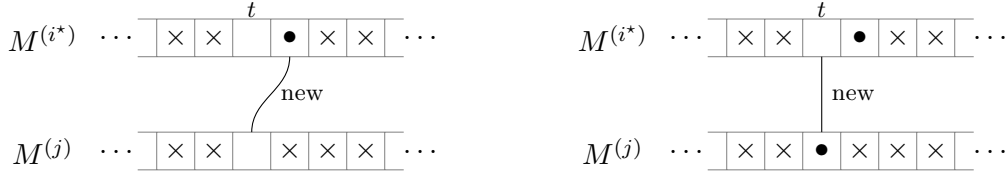
\begin{figure}[t]
	\centering
	\begin{tikzpicture}[scale=0.5]
	\draw[very thin,color=gray] (4.5,1) grid (11.5,0);
	\draw (4,0.5) node {$\dots$};
	\draw (12,0.5) node {$\dots$};
	\draw (2,0.5) node {$M^{(i^{\star})}$};
	\foreach \x in {-2,-1,2,3}
	{
		\draw (\x+0.5+7,0.5) node {$\times$};
	}
	\foreach \x in {1}
	{
		\draw (\x+0.5+7,0.5) node {$\bullet$};
	}
	\draw[] (7.5,1.25) node {\footnotesize{$t$}};
	
	\draw[very thin,color=gray] (4.5,-2) grid (11.5,-3);
	\draw (4,-2.5) node {$\dots$};
	\draw (12,-2.5) node {$\dots$};
	\draw (2,-2.5) node {$M^{(j)}$};
	\foreach \x in {-2,-1,1,2,3}
	{
		\draw (\x+0.5+7,-2.5) node {$\times$};
	}
	\foreach \x in {}
	{
		\draw (\x+0.5+7,-2.5) node {$\bullet$};
	}
	
	\draw[] (7.5,-2) to[out=90,in=-90] node[pos=0.5,right]{\footnotesize{new}} (8.5,0);
	
	\draw[very thin,color=gray] (4.5+15,1) grid (11.5+15,0);
	\draw (4+15,0.5) node {$\dots$};
	\draw (12+15,0.5) node {$\dots$};
	\draw (2+15,0.5) node {$M^{(i^{\star})}$};
	\foreach \x in {-2,-1,2,3}
	{
		\draw (\x+0.5+7+15,0.5) node {$\times$};
	}
	\foreach \x in {1}
	{
		\draw (\x+0.5+7+15,0.5) node {$\bullet$};
	}
	\draw[] (7.5+15,1.25) node {\footnotesize{$t$}};
	
	\draw[very thin,color=gray] (4.5+15,-2) grid (11.5+15,-3);
	\draw (4+15,-2.5) node {$\dots$};
	\draw (12+15,-2.5) node {$\dots$};
	\draw (2+15,-2.5) node {$M^{(j)}$};
	\foreach \x in {-2,-1,1,2,3}
	{
		\draw (\x+0.5+7+15,-2.5) node {$\times$};
	}
	\foreach \x in {0}
	{
		\draw (\x+0.5+7+15,-2.5) node {$\bullet$};
	}
	
	\draw[] (7.5+15,-2) to[out=90,in=-90] node[pos=0.5,right]{\footnotesize{new}} (7.5+15,0);
	\end{tikzpicture}
	\caption{New factors for~$j>i^{\star}$.}
	\label{fig:NewFactors2}
\end{figure} 
	
The overall conclusion is that the difference of the (leading) behavior in~$k$ between the fractions  $H_{\text{non-$p$-fold}}(\lambda(k))/H(\bar{\lambda}(k))$ and $H_{\text{non-$p$-fold}}(\tilde{\lambda}(k))/H(\bar{\lambda}(k))$ is given by~\eqref{eq:ExtraFactors}. Hence, if we recall $|\tilde{\mu}|=|\mu|-1$, we get that the coefficient associated to~$k^{(p-1)|\mu|}$ in $H_{\text{non-$p$-fold}}(\lambda(k)))/H(\bar{\lambda}(k))$ equals the coefficient associated to~$k^{(p-1)|\tilde{\mu}|}$ in $H_{\text{non-$p$-fold}}(\tilde{\lambda}(k)))/H(\bar{\lambda}(k))$ times
\begin{equation}\label{eq:ProofFraction2}
	p^{p-1} \prod_{\substack{j=0\\ j\neq i^{\star}}}^{p-1} \big( (-1)^{h_{i^{\star},j}} (a_{i^{\star}}-a_j) \big)
		= (-1)^{d} \, p^{p-1} \alpha_{i^{\star}}
\end{equation}
with~$\alpha_{i^{\star}}=\prod_{j\neq i^{\star}} (a_{i^{\star}}-a_j)$ and
\begin{equation}\label{eq:Definitiond}
	d
		= \# \{j < i^{\star} \mid t+1\in M^{(j)}\} + \# \{j > i^{\star} \mid t\in M^{(j)}\}
\end{equation}
which indicates the number of new factors that are obtained from the empty box located at~$t$ in~$M^{(i^{\star})}$. Applying the induction and recalling that $|\tilde{\mu}^{(i^{\star})}|=|\mu^{(i^{\star})}|-1$ and $|\tilde{\mu}^{(i)}|=|\mu^{(i)}|$ for~$i\neq i^{\star}$, as well as using~\eqref{eq:ProofFraction2}, yields that the coefficient associated to~$k^{(p-1)|\mu|}$ is given by
\begin{equation}\label{eq:ProofFraction2.5}
	(-1)^{d+\htt_p(\tilde{\lambda}(k)/\bar{\lambda}(k))} \, p^{(p-1)|\mu|} \, \prod_{i=0}^{p-1}\alpha_{i}^{|\mu^{(i)}|}.
\end{equation}
We therefore obtain the result if
\begin{equation}\label{eq:ProofFraction3}
	d
		\equiv \htt(\lambda(k)/\tilde{\lambda}(k)) \mod 2.
\end{equation}
To show this, recall that $\tilde{\mu}<_1\mu$ and so the difference between both Young diagrams of~$\lambda(k)$ and~$\tilde{\lambda}(k)$ is a border strip of size~$p$. In terms of the Maya diagrams~$M_{\tilde{\lambda}(k)}$ and~$M_{\lambda(k)}$, a bullet at place~$T$ in~$M_{\tilde{\lambda}(k)}$ moves~$p$ steps to the right, and thereby passing exactly~$d$ filled other boxes, see Figure~\ref{fig:JumpFilledBox}. This exactly means that $d=\htt(\lambda(k)/\tilde{\lambda}(k))$, see Lemma~\ref{lem:HeightMayaDiagram}, and thereby confirms~\eqref{eq:ProofFraction3}.
\end{proof}

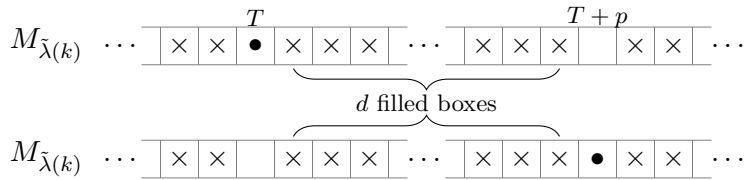
\begin{figure}[!h]
	\centering
	\begin{tikzpicture}[scale=0.5]
		\draw[very thin,color=gray] (-1.5,1) grid (5.5,0);
		\draw (-2,0.5) node {$\dots$};
		\draw (-4,0.5) node {$M_{\tilde{\lambda}(k)}$};
		\foreach \x in {-8,-7,-5,-4,-3}
		{
			\draw (\x+0.5+7,0.5) node {$\times$};
		}
		\draw (7+0.5-6,0.5) node {$\bullet$};
		\draw (6,0.5) node {$\dots$};
		\draw[very thin,color=gray] (6.5,1) grid (13.5,0);
		\foreach \x in {0,1,2,4,5}
		{
			\draw (\x+0.5+7,0.5) node {$\times$};
		}
		\draw (14,0.5) node {$\dots$};
		\draw[] (1.5,1.25) node {\footnotesize{$T$}};
		\draw[] (10.5,1.25) node {\footnotesize{$T+p$}};
		\draw [decorate,decoration={brace,amplitude=8pt,mirror}] (2.5,-0.1) -- (9.5,-0.1) node [black,midway,yshift=-13pt] 
		{\footnotesize $d$ filled boxes};
		
		\draw[very thin,color=gray] (-1.5,-2) grid (5.5,-3);
		\draw (-2,-2.5) node {$\dots$};
		\draw (-4,-2.5) node {$M_{\tilde{\lambda}(k)}$};
		\foreach \x in {-8,-7,-5,-4,-3}
		{
			\draw (\x+0.5+7,-2.5) node {$\times$};
		}
		\draw (7+0.5+3,-2.5) node {$\bullet$};
		\draw (6,-2.5) node {$\dots$};
		\draw[very thin,color=gray] (6.5,-2) grid (13.5,-3);
		\foreach \x in {0,1,2,4,5}
		{
			\draw (\x+0.5+7,-2.5) node {$\times$};
		}
		\draw (14,-2.5) node {$\dots$};
		\draw [decorate,decoration={brace,amplitude=8pt}] (2.5,-1.9) -- (9.5,-1.9) node [black,midway,yshift=-15pt] 
		{};
	\end{tikzpicture}
	\caption{The movement of the filled box.}
	\label{fig:JumpFilledBox}
\end{figure}

\begin{example}\label{ex:HookLengthsProof}
We show the main idea of the previous proof using the partitions $\lambda=(6,5,4,1,1)$ and $\tilde{\lambda}=(4,4,4,1,1)$. Both partitions have $3$-core~$(1,1)$, or equivalently $c_{\lambda}=c_{\tilde{\lambda}}=(1,-1,0)$. The $3$-quotient is different for both partitions, we have $\mu=((1),(2),(2))$ while $\tilde{\mu}=((1),(2),(1))$, and so $\tilde{\mu}<_1 \mu$. More precisely, in Figure~\ref{fig:HookLengthsQuotientProof} we have that $\widetilde{M}^{(0)}=M^{(0)}$, $\widetilde{M}^{(1)}=M^{(1)}$, but $\widetilde{M}^{(2)}\neq M^{(2)}$. In these last two different Maya diagrams, a bullet is moved 1 step to the right, and therefore the Maya diagram~$\widehat{M}_{\lambda}$ can be obtained from~$\widehat{M}_{\tilde{\lambda}}$ by moving a bullet 3 steps to the right. The associated boxes are indicated with white bullets. Note that by the movement of the white bullet, it passes 1 black bullet which equals the height of the border strip $\lambda/\tilde{\lambda}$. It also implies that 1 of the new factors, as described in the previous proof, is obtained from an empty box in~$M^{(2)}$, while the other new one is obtained from a filled box in~$M^{(2)}$.
	
\begin{figure}[!h]
	\centering
	\begin{tikzpicture}[scale=0.35]
		\foreach \x in {-6,-3,0,3,6}
		{
			\draw[fill=black!08!white,draw=none] (7+\x,1) rectangle (7+1+\x,0);
		}
		\foreach \x in {-7,-4,-1,2,5}
		{
			\draw[fill=black!32!white,draw=none] (7+\x,1) rectangle (7+1+\x,0);
		}
		\foreach \x in {-8,-5,-2,1,4,7}
		{
			\draw[fill=black!20!white,draw=none] (7+\x,1) rectangle (7+1+\x,0);
		}
		\draw[fill=black!32!white,draw=none] (7+8,1) rectangle (7+0.5+8,0);
		\draw[fill=black!08!white,draw=none] (7-9+0.5,1) rectangle (7-8,0);
		\draw[very thin,color=gray] (-1.5,1) grid (15.5,0);
		\draw[thick,color=black] (7,1.4) -- (7,-0.4);
		\draw (-2,0.5) node {$\dots$};
		\draw (16.25,0.5) node {$\dots$};
		\draw (-4,0.5) node {$\widehat{M}_{\tilde{\lambda}}$};
		\foreach \x in {-8,-7,-6,-4,-3,1,,3}
		{
			\draw (\x+7+0.5,0.5) node {$\bullet$};
		}
		\draw[white] (2+7+0.5,0.5) node {$\bullet$};
		
		\draw[fill=black!08!white,draw=none] (2.5,-2) rectangle (11.5,-3);
		\draw[very thin,color=gray] (2.5,-2) grid (11.5,-3);
		\draw[thick,color=black] (7,-1.6) -- (7,-3.4);
		\draw[thick,color=black,dotted] (8,-1.6) -- (8,-3.4);
		\draw (2,-2.5) node {$\dots$};
		\draw (12.25,-2.5) node {$\dots$};
		\draw (0,-2.5) node {$\widetilde{M}^{(0)}$};
		\foreach \x in {-4,-3,-2,-1,1}
		{
			\draw (\x+7+0.5,-2.5) node {$\bullet$};
		}
		
		\draw[fill=black!20!white,draw=none] (2.5,-5) rectangle (11.5,-6);
		\draw[very thin,color=gray] (2.5,-5) grid (11.5,-6);
		\draw[thick,color=black] (7,-4.6) -- (7,-6.4);
		\draw[thick,color=black,dotted] (6,-4.6) -- (6,-6.4);
		\draw (2,-5.5) node {$\dots$};
		\draw (12.25,-5.5) node {$\dots$};
		\draw (0,-5.5) node {$\widetilde{M}^{(1)}$};
		\foreach \x in {-4,-3,0}
		{
			\draw (\x+7+0.5,-5.5) node {$\bullet$};
		}
		
		\draw[fill=black!32!white,draw=none] (2.5,-8) rectangle (11.5,-9);
		\draw[very thin,color=gray] (2.5,-8) grid (11.5,-9);
		\draw[thick,color=black] (7,-7.6) -- (7,-9.4);
		\draw (2,-8.5) node {$\dots$};
		\draw (12.25,-8.5) node {$\dots$};
		\draw (0,-8.5) node {$\widetilde{M}^{(2)}$};
		\foreach \x in {-4,-3,-2}
		{
			\draw (\x+7+0.5,-8.5) node {$\bullet$};
		}	
		\draw[white] (0+7+0.5,-8.5) node {$\bullet$};
	\end{tikzpicture}
	\begin{tikzpicture}[scale=0.35]
		\foreach \x in {-6,-3,0,3,6}
		{
			\draw[fill=black!08!white,draw=none] (7+\x,1) rectangle (7+1+\x,0);
		}
		\foreach \x in {-7,-4,-1,2,5}
		{
			\draw[fill=black!32!white,draw=none] (7+\x,1) rectangle (7+1+\x,0);
		}
		\foreach \x in {-8,-5,-2,1,4,7}
		{
			\draw[fill=black!20!white,draw=none] (7+\x,1) rectangle (7+1+\x,0);
		}
		\draw[fill=black!32!white,draw=none] (7+8,1) rectangle (7+0.5+8,0);
		\draw[fill=black!08!white,draw=none] (7-9+0.5,1) rectangle (7-8,0);
		\draw[very thin,color=gray] (-1.5,1) grid (15.5,0);
		\draw[thick,color=black] (7,1.4) -- (7,-0.4);
		\draw (-2,0.5) node {$\dots$};
		\draw (16.25,0.5) node {$\dots$};
		\draw (-4,0.5) node {$\widehat{M}_{\lambda}$};
		\foreach \x in {-8,-7,-6,-4,-3,1,3}
		{
			\draw (\x+0.5+7,0.5) node {$\bullet$};
		}
		\draw[white] (5+7+0.5,0.5) node {$\bullet$};

		\draw[fill=black!08!white,draw=none] (2.5,-2) rectangle (11.5,-3);
		\draw[very thin,color=gray] (2.5,-2) grid (11.5,-3);
		\draw[thick,color=black] (7,-1.6) -- (7,-3.4);
		\draw[thick,color=black,dotted] (8,-1.6) -- (8,-3.4);
		\draw (2,-2.5) node {$\dots$};
		\draw (12.25,-2.5) node {$\dots$};
		\draw (0,-2.5) node {$M^{(0)}$};
		\foreach \x in {-4,-3,-2,-1,1}
		{
			\draw (\x+7+0.5,-2.5) node {$\bullet$};
		}
		
		\draw[fill=black!20!white,draw=none] (2.5,-5) rectangle (11.5,-6);
		\draw[very thin,color=gray] (2.5,-5) grid (11.5,-6);
		\draw[thick,color=black] (7,-4.6) -- (7,-6.4);
		\draw[thick,color=black,dotted] (6,-4.6) -- (6,-6.4);
		\draw (2,-5.5) node {$\dots$};
		\draw (12.25,-5.5) node {$\dots$};
		\draw (0,-5.5) node {$M^{(1)}$};
		\foreach \x in {-4,-3,0}
		{
			\draw (\x+7+0.5,-5.5) node {$\bullet$};
		}
		
		\draw[fill=black!32!white,draw=none] (2.5,-8) rectangle (11.5,-9);
		\draw[very thin,color=gray] (2.5,-8) grid (11.5,-9);
		\draw[thick,color=black] (7,-7.6) -- (7,-9.4);
		\draw (2,-8.5) node {$\dots$};
		\draw (12.25,-8.5) node {$\dots$};
		\draw (0,-8.5) node {$M^{(2)}$};
		\foreach \x in {-4,-3,-2}
		{
			\draw (\x+7+0.5,-8.5) node {$\bullet$};
		}	
		\draw[white] (1+7+0.5,-8.5) node {$\bullet$};
		\draw[gray] (7.5,-8) to[out=90,in=-90] node[pos=0.8,left]{\footnotesize{{\color{gray} new}}} (8.5,-3);
		\draw[gray] (8.5,-8) to[out=90,in=-90] node[pos=0.5,right]{\footnotesize{{\color{gray} new}}} (8.5,-6);
	\end{tikzpicture}
	\caption{The new hook lengths represented in the Maya diagrams of the $3$-quotient (Example~\ref{ex:HookLengthsProof}).}
	\label{fig:HookLengthsQuotientProof}
\end{figure}
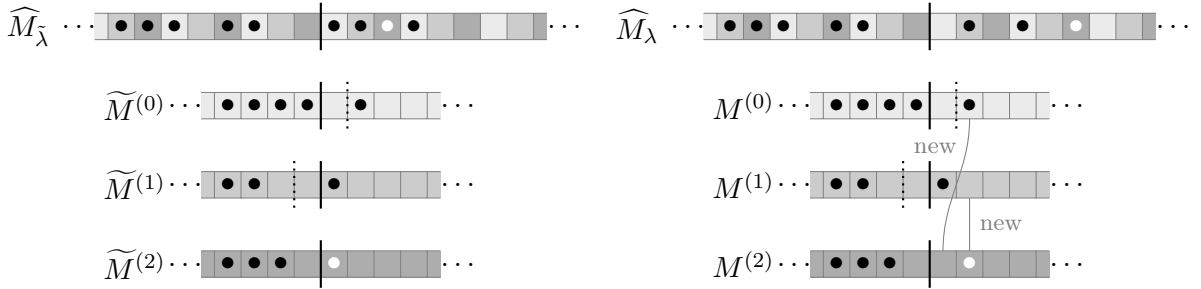
\end{example}

\subsection{Asymptotic behavior of the coefficients}\label{sec:BehaviorCoefficients}
In this section we use Lemma~\ref{lem:HookLengths} to obtain the behavior of~$r_{\lambda(k),j}$ as indicated in~\eqref{eq:CoefficientLeadingBehavior}. As before, let~$\lambda(k)$ be the partition identified with the $p$-quotient $\mu$ and the $p$-core associated to~$c(k)$.

\begin{lemma}\label{lem:BehaviorCoefficient}
For each~$j$, the coefficient~$r_{\lambda(k),j}$ described in~\eqref{eq:Coefficient} satisfies 
\begin{equation}\label{eq:BehaviorCoefficient}
	r_{\lambda(k),j} 
		= L_j (kp)^{(p-1)j} + o(k^{(p-1)j})
\end{equation}
as~$k\to + \infty$, and where~$L_j$ is defined in~\eqref{eq:L}.
\end{lemma}
\begin{proof}
Recall that
\begin{equation}\label{eq:ProofBehaviorCoefficient1}
	r_{\lambda(k),j}
		= (-1)^j \binom{|\mu|}{j} 
		\sum_{\tilde{\mu}<_j \mu} (-1)^{\htt_p(\lambda(k)/\tilde{\lambda}(k))} 
		\frac{F^{(p)}_{\tilde{\mu}} \,F^{(p)}_{\mu/\tilde{\mu}}}{F^{(p)}_{\mu}} 
		\,
		\frac{H_{\text{non-$p$-fold}}(\lambda(k))}{H_{\text{non-$p$-fold}}(\tilde{\lambda}(k))}
\end{equation}
where the partition~$\lambda(k)$ (respectively~$\tilde{\lambda}(k)$) is identified with the $p$-quotient~$\mu$ (respectively~$\tilde{\mu}$) and the $p$-core associated to~$c(k)$. In the above expression, we sum over all $p$-tuples of partitions $\tilde{\mu}=(\tilde{\mu}^{(0)},\tilde{\mu}^{(1)},\dots,\tilde{\mu}^{(p-1)})$ such that $\tilde{\mu}^{(i)} \leq \mu^{(i)}$ for all $i$, and $\sum_i \left(|\mu^{(i)}|-|\tilde{\mu}^{(i)}|\right) = j$. Set $l_i=|\mu^{(i)}|-|\tilde{\mu}^{(i)}|$ and rewrite \eqref{eq:ProofBehaviorCoefficient1} as
\begin{multline}\label{eq:ProofBehaviorCoefficient2}
	r_{\lambda(k),j}
		= (-1)^j \binom{|\mu|}{j} 
		\sum_{l_0+\cdots+l_{p-1}=j}
		\sum_{\tilde{\mu}^{(0)}<_{l_0} \mu^{(0)}} 
		\cdots
		\sum_{\tilde{\mu}^{(p-1)}<_{l_{p-1}} \mu^{(p-1)}} 
		(-1)^{\htt_p(\lambda(k)/\tilde{\lambda}(k))} 
		\\
		\times
		\frac{F^{(p)}_{\tilde{\mu}} \,F^{(p)}_{\mu/\tilde{\mu}}}{F^{(p)}_{\mu}} 
		\,
		\frac{H_{\text{non-$p$-fold}}(\lambda(k))}{H_{\text{non-$p$-fold}}(\tilde{\lambda}(k))}
\end{multline}
where we now sum over all non-negative integers $l_0,l_1,\dots,l_{p-1}$ such that~$\sum_i l_i=j$.
	
As~$\lambda(k)$ and~$\tilde{\lambda}(k)$ have the same $p$-core, we can apply Lemma~\ref{lem:HookLengths} to write the last fraction in~\eqref{eq:ProofBehaviorCoefficient2} as
\begin{equation}\label{eq:ProofBehaviorCoefficient3}
	\frac{H_{\text{non-$p$-fold}}(\lambda(k))}{H_{\text{non-$p$-fold}}(\tilde{\lambda}(k))}
		= (-1)^{\htt_p(\lambda(k)/\tilde{\lambda}(k))} (pk)^{(p-1)j}\prod_{i=0}^{p-1}\alpha_{i}^{|\mu^{(i)}|-|\tilde{\mu^{(i)}}|}
		+ o( k^{(p-1)j})
\end{equation}
as~$k\to + \infty$, and where we used that
\begin{equation}\label{eq:IdentityHeight}
	\htt_p(\lambda(k)/\bar{\lambda}(k)) - \htt_p(\tilde{\lambda}(k)/\bar{\lambda}(k))
		\equiv \htt_p(\lambda(k)/\tilde{\lambda}(k)) \mod 2.
\end{equation}
Combining \eqref{eq:ProofBehaviorCoefficient2} and \eqref{eq:ProofBehaviorCoefficient3} yields 
\begin{multline}\label{eq:ProofBehaviorCoefficient3.5}
	(pk)^{-(p-1)j} \, r_{\lambda(k),j}
		= (-1)^j \binom{|\mu|}{j} 
		\sum_{l_0+\cdots+l_{p-1}=j}
		\alpha_0^{l_0} \alpha_1^{l_1} \cdots \alpha_{p-1}^{l_{p-1}} 
		\\
		\times
		\sum_{\tilde{\mu}^{(0)}<_{l_0} \mu^{(0)}} 
		\cdots
		\sum_{\tilde{\mu}^{(p-1)}<_{l_{p-1}} \mu^{(p-1)}} 
		\frac{F^{(p)}_{\tilde{\mu}} \,F^{(p)}_{\mu/\tilde{\mu}}}{F^{(p)}_{\mu}} + o(1)
\end{multline}
as~$k\to+\infty$. If we simplify the binomial coefficient and the numbers in the last fraction using~\eqref{eq:Fp}, we get that the right-hand side equals
\begin{multline}\label{eq:ProofBehaviorCoefficient4}
	(-1)^j \sum_{l_0+\cdots+l_{p-1}=j}
	\alpha_0^{l_0} \alpha_1^{l_1} \cdots \alpha_{p-1}^{l_{p-1}} \binom{|\mu^{(0)}|}{l_0} \binom{|\mu^{(1)}|}{l_1} \cdots \binom{|\mu^{(p-1)}|}{l_{p-1}}
	\\
	\times
	\left(\sum_{\tilde{\mu}^{(0)}<_{l_0} \mu^{(0)}} \frac{F_{\tilde{\mu}^{(0)}} F_{\mu^{(0)}/\tilde{\mu}^{(0)}}}{F_{\mu^{(0)}}}\right)
	\cdots
	\left(\sum_{\tilde{\mu}^{(p-1)}<_{l_{p-1}} \mu^{(p-1)}} \frac{F_{\tilde{\mu}^{(p-1)}} F_{\mu^{(p-1)}/\tilde{\mu}^{(p-1)}}}{F_{\mu^{(p-1)}}}\right) + o(1).
\end{multline}
From~\eqref{eq:IdentityCountingPaths} we find that each sum in the second line of~\eqref{eq:ProofBehaviorCoefficient4} equals 1. Hence, this yields
\begin{equation}\label{eq:ProofBehaviorCoefficient5}
	r_{\lambda(k),j}
		= L_j (pk)^{(p-1)j} + o( k^{(p-1)j} )
\end{equation}
as~$k\to + \infty$, and where~$L_j$ is defined in~\eqref{eq:L}. This ends the proof.
\end{proof}

\subsection{Asymptotic behavior of the polynomial}\label{sec:AsymptoticBehavior}
Via Lemma~\ref{lem:BehaviorCoefficient} we now obtain the asymptotic result stated in~\eqref{eq:AsymptoticResult}.

\begin{proof}[Proof of Theorem~\ref{thm:AsymptoticResult}]
We first determine the behavior of the fraction in the left-hand side of~\eqref{eq:AsymptoticResult}, and then take the limit to find the result. Via the expansion~\eqref{eq:ExpansionR} we get
\begin{equation}\label{eq:ProofAsymptoticResult1}
	\frac{R_{\lambda(k)}\left((pk)^{p-1}x\right)}{(pk)^{(p-1)|\mu|}}
		= \sum_{j=0}^{|\mu|} \frac{r_{\lambda(k),j}}{(pk)^{(p-1)j}} \, x^{|\mu|-j}.
\end{equation}
We now use Lemma~\ref{lem:BehaviorCoefficient} to describe the behavior of the coefficient~$r_{\lambda(k),j}$. This yields
\begin{equation}\label{eq:ProofAsymptoticResult2}
	\frac{R_{\lambda(k)}\left((pk)^{p-1}x\right)}{(pk)^{(p-1)|\mu|}}
		= \sum_{j=0}^{|\mu|} \left(L_j+o(1)\right) \, x^{|\mu|-j}
\end{equation}
as~$k\to + \infty$, with~$L_j$ defined in~\eqref{eq:L}. Applying the limit gives
\begin{equation}\label{eq:ProofAsymptoticResult3}
	\lim_{k\to + \infty} \frac{R_{\lambda(k)}\left((pk)^{p-1}x\right)}{(pk)^{(p-1)|\mu|}}
		= \sum_{j=0}^{|\mu|} L_j \, x^{|\mu|-j}
\end{equation}
which equals the right-hand side of~\eqref{eq:AsymptoticResult} because of~\eqref{eq:BinomialExpansion}. This ends the proof.
\end{proof}

\begin{remark}
In the special case that all entries of the characteristic vector are multiples of~$k$, i.e., $c(k)=(a_0k, a_1k,\dots,a_{p-1}k)$ with $a_0,a_1,\dots,a_{p-1}\in\mathbb{Z}$ and $\sum_i a_i=0$, then Lemma~\ref{lem:BehaviorCoefficient} results in
\begin{equation}\label{eq:ResultSpecialCase}
	r_{\lambda(k),j} 
		= L_j (kp)^{(p-1)j} + O(k^{(p-1)j-1}).
\end{equation} 
This then implies that the speed of convergence in Theorem~\ref{thm:AsymptoticResult} is~$O(k^{-1})$, and that all zeros are attracted by integer values. An example is illustrated in Figure~\ref{fig:ConvergenceZeros}.
\end{remark}

\subsection{\texorpdfstring{The size of the $p$-core expressed via the characteristic vector}{The size of the p-core expressed via the characteristic vector}}\label{sec:SizeCore}

Each $p$-core is uniquely described via its characteristic vector. We now prove the expression for the size of the $p$-core in terms of the entries from the characteristic vector as stated in Remark~\ref{rem:SizeCore}. The proof is based on the ideas in Section~\ref{sec:HookLengths} whereas the result itself was already obtained in~\cite{Garvan_Kim_Stanton}. 

\begin{proposition}
For any positive integer~$p$ and any $p$-core $\bar{\lambda}$ described via the characteristic vector $(c_0,c_1,\dots,c_{p-1})$, we have that
\begin{equation}\label{eq:SizeCoreProof}
	|\bar{\lambda}|
		= \frac{p}{2} \sum_{j=0}^{p-1} c_j^2 + \sum_{j=1}^{p-1} j c_j.
\end{equation}	
\end{proposition}
\begin{proof}
The size of any partition equals the number of boxes in the Young diagram. Each box in the diagram has an associated hook length, hence the number of factors in the product over all hook lengths equals the size of the partition. To count the number of factors in~$H(\bar{\lambda})$, we express the hook lengths of~$\bar{\lambda}$ via the Maya diagrams of the quotient~$\widebar{M}^{(0)},\widebar{M}^{(1)},\dots,\widebar{M}^{(p-1)}$, that is
\begin{equation}\label{eq:HookLengthsMayaDiagramQuotient}
	H(\bar{\lambda})
		= \prod_{i=0}^{p-1} \prod_{j=0}^{p-1}\prod_{\substack{m\in \widebar{M}^{(i)} \\ n\notin \widebar{M}^{(j)} \\ pm+i>pn+j}} \big((pm+i)-(pn+j)\big)
\end{equation}
where $\widebar{M}^{(i)}=M_{\emptyset}+c_i$ for~$i=0,1,\dots,p-1$. As each Maya diagram is equivalent to~$M_{\emptyset}$, the last product in~\eqref{eq:HookLengthsMayaDiagramQuotient} has no factors when~$i=j$. We now argue that the last product of~\eqref{eq:HookLengthsMayaDiagramQuotient} has $(c_i-c_j)(c_i-c_j-1)/2$ factors when~$i\neq j$. We distinguish four possibilities.
\begin{enumerate}
	\item $i<j$ and $c_i>c_j$
	\item $i<j$ and $c_i\leq c_j$
	\item $j>i$ and $c_i<c_j$
	\item $j>i$ and $c_i\geq c_j$
\end{enumerate} 
The first possibility is visualized in Figure~\ref{fig:NumberConnectionsCore}. We directly observe that for each filled box in~$\widebar{M}^{(i)}$ located at~$t$ with $t\in\{c_j+1,c_j+2,\dots,c_i\}$, the empty boxes located at $t-1,t-2,\dots,c_j+1$ in~$\widebar{M}^{(j)}$ give rise to a factor in~$H(\bar{\lambda})$. Counting all possibilities directly gives the number of~$(c_i-c_j)(c_i-c_j-1)/2$ factors. The other three possibilities can be treated similarly and are left out. 
	
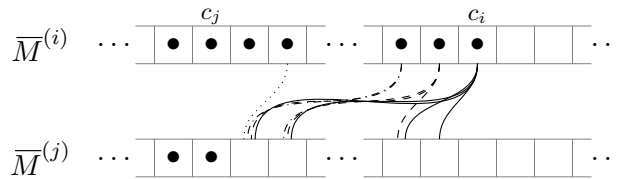
\begin{figure}[b]
	\centering
	\begin{tikzpicture}[scale=0.5]
		\draw[very thin,color=gray] (-0.5,1) grid (4.5,0);	
		\draw[very thin,color=gray] (5.5,1) grid (11.5,0);
		\draw (-1,0.5) node {$\dots$};
		\draw (-3,0.5) node {$\widebar{M}^{(i)}$};
		\foreach \x in {-7,-6,-5,-4,-1,0,1}
		{
			\draw (\x+0.5+7,0.5) node {$\bullet$};
		}
		\draw (5,0.5) node {$\dots$};
		\draw (12,0.5) node {$\dots$};
		\draw[] (1.5,1.25) node {\footnotesize{$c_j$}};
		\draw[] (8.5,1.25) node {\footnotesize{$c_i$}};

		\draw[very thin,color=gray] (-0.5,-2) grid (4.5,-3);
		\draw[very thin,color=gray] (5.5,-2) grid (11.5,-3);
		\draw (-1,-2.5) node {$\dots$};
		\draw (-3,-2.5) node {$\widebar{M}^{(j)}$};
		\foreach \x in {-7,-6}
		{
			\draw (\x+0.5+7,-2.5) node {$\bullet$};
		}
		\draw (5,-2.5) node {$\dots$};
		\draw (12,-2.5) node {$\dots$};
		
		\draw[] (8.5,0) to[out=-90,in=90] (7.5,-2);
		\draw[] (8.5,0) to[out=-90,in=90] (6.6,-2);
		\draw[] (8.5,0) to[out=-90,in=90] (3.6,-2);
		\draw[] (8.5,0) to[out=-90,in=90] (2.65,-2);
		
		\draw[dashed] (7.5,0) to[out=-90,in=90] (6.4,-2);
		\draw[dashed] (7.5,0) to[out=-90,in=90] (3.5,-2);
		\draw[dashed] (7.5,0) to[out=-90,in=90] (2.55,-2);
		
		\draw[dashdotted] (6.5,0) to[out=-90,in=90] (3.4,-2);
		\draw[dashdotted] (6.5,0) to[out=-90,in=90] (2.45,-2);
		
		\draw[dotted] (3.5,0) to[out=-90,in=90] (2.35,-2);
	\end{tikzpicture}
	\caption{Counting the number of hook lengths.}
	\label{fig:NumberConnectionsCore}
\end{figure}
	
As each pair~$(i,j)$ in~\eqref{eq:HookLengthsMayaDiagramQuotient} leads to $(c_i-c_j)(c_i-c_j-1)/2$ factors, and as the total number of factors in~$H(\bar{\lambda})$ equals $|\bar{\lambda}|$, we find that
\begin{equation}\label{eq:SizeCoreAlternativeExpression}
	|\bar{\lambda}|
		= \frac{1}{2}\sum_{0\leq i < j \leq p-1} (c_i-c_j)(c_i-c_j-1).
\end{equation}
We now rewrite this expression to obtain~\eqref{eq:SizeCoreProof}. Expanding the sum gives
\begin{equation}\label{eq:Expansion00}
	|\bar{\lambda}|
		= \frac{1}{2} \sum_{0\leq i < j \leq p-1} (c_i^2+c_j^2)
		+ \frac{1}{2} \sum_{0\leq i < j \leq p-1} (c_j-c_i)
		- \sum_{0\leq i < j \leq p-1} c_i \, c_j.
\end{equation}
An elementary calculation shows that the first two sums can be rewritten as
\begin{align}
	\sum_{0\leq i < j \leq p-1} (c_i^2+c_j^2)
		&= (p-1) \sum_{j=0}^{p-1} c_j^2, 
	\label{eq:Expansion0} \\
	\sum_{0\leq i < j \leq p-1} (c_j-c_i)
		&= 2 \sum_{j=0}^{p-1} j c_j + (1-p) \sum_{j=0}^{p-1} c_j.
	\label{eq:Expansion1}
\end{align}
Now we use that~$c_\lambda$ is a characteristic vector, i.e., $\sum_{i}c_i=0$, so that the last sum in~\eqref{eq:Expansion1} vanishes. Taking the square of both sides in the condition yields the identity
\begin{equation}\label{eq:Expansion2}
	\sum_{0\leq i < j \leq p-1} c_i \, c_j
		= -\frac{1}{2} \sum_{j=0}^{p-1} c_j^2.
\end{equation}
Hence, if we use~\eqref{eq:Expansion0}, \eqref{eq:Expansion1} and~\eqref{eq:Expansion2} to rewrite~\eqref{eq:Expansion00}, we get~\eqref{eq:SizeCoreProof}.
\end{proof}

\section*{Acknowledgments}
The author is thankful to Dan Betea, Arno Kuijlaars and Marco Stevens for their valuable feedback and carefully reading several versions of this paper. This work is supported in part by the long term structural funding~-- Methusalem grant of the Flemish Government, and by EOS project 30889451 of the Flemish Science Foundation (FWO).

\end{document}